\documentclass[reqno,english]{amsart}
\usepackage{amsfonts,amsmath,latexsym,verbatim,amscd,mathrsfs,color,array}
\usepackage[colorlinks=true]{hyperref}
\usepackage{amssymb,amsthm,graphicx,color}
\usepackage[hmargin=3.6cm, vmargin=3.6cm]{geometry}
\usepackage{bbm}
\usepackage{amssymb}
\usepackage{pdfsync}
\usepackage{epstopdf}
\usepackage{cite}
\usepackage[dvipsnames]{xcolor}
\usepackage{graphicx}
\usepackage{multirow}
\usepackage[font=bf,aboveskip=15pt]{caption}
\usepackage[toc,page]{appendix}
\usepackage[colorlinks=true]{hyperref}
\usepackage{bookmark}
\DeclareFontShape{OT1}{cmr}{bx}{sc}{<-> cmbcsc10}{}
\newcommand{\R}{\mathbb R}
\newcommand{\al}{\alpha}

\newcommand{\C}{\mathbb C}

\newcommand{\bt}{\begin{theorem}}
\newcommand{\et}{\end{theorem}}
\newcommand{\bl}{\begin{lemma}}
\newcommand{\el}{\end{lemma}}
\newcommand{\bd}{\begin{definition}}
\newcommand{\ed}{\end{definition}}
\newcommand{\bc}{\begin{corollary}}
\newcommand{\ec}{\end{corollary}}
\newcommand{\bp}{\begin{proof}}
\newcommand{\ep}{\end{proof}}
\newcommand{\bx}{\begin{example}}
\newcommand{\ex}{\end{example}}
\newcommand{\bi}{\begin{exercise}}
\newcommand{\ei}{\end{exercise}}
\newcommand{\bo}{\begin{prop}}
\newcommand{\eo}{\end{prop}}
\newcommand{\br}{\begin{remark}}
\newcommand{\er}{\end{remark}}
\newcommand{\be}{\begin{equation}}
\newcommand{\ee}{\end{equation}}
\newcommand{\ba}{\begin{align}}
\newcommand{\ea}{\end{align}}
\newcommand{\bn}{\begin{enumerate}}
\newcommand{\en}{\end{enumerate}}
\newcommand{\bg}{\begin{align*}}
\newcommand{\bcs}{\begin{cases}}
\newcommand{\ecs}{\end{cases}}
\newcommand{\abs}[1]{\left|#1\right|}

\newcommand{\ve}{\varepsilon  }

\newcommand{\bean}{\begin{eqnarray*}}
\newcommand{\eean}{\end{eqnarray*}}

\usepackage{theoremref}
\usepackage{mathrsfs}
\usepackage{comment}
\usepackage{cancel}

\newcommand{\xMapsto}[2][]{\ext@arrow 0599{\Mapstofill@}{#1}{#2}}
\def\Mapstofill@{\arrowfill@{\Mapstochar\Relbar}\Relbar\Rightarrow}
\makeatother

\newcommand{\bke}[1]{\left ( #1 \right )}
\newcommand{\bkt}[1]{\left [ #1 \right ]}
\newcommand{\bket}[1]{\left \{ #1 \right \}}
\newcommand{\norm}[1]{ \left\| #1  \right\|}

\renewcommand{\div}{\mathop{\rm div}}

\newcommand{\pd}{\partial}

\newcommand{\si}{\sigma}

\newcommand\De{\Delta}
\newcommand\de{\delta}
\newcommand{\na}{\nabla}
\newcommand{\lec}{{\ \lesssim \ }}

\newcommand{\EQ}[1]{\begin{equation}\begin{split} #1 \end{split}\end{equation}}
\newcommand{\EQN}[1]{\begin{equation*}\begin{split} #1 \end{split}\end{equation*}}
\newcommand{\EN}[1]{\begin{enumerate} #1 \end{enumerate}}

\newcommand{\calE}{{ \mathcal E  }}

\numberwithin{equation}{section}
\newtheorem{thm}{Theorem}[section]

\newtheorem{lem}[thm]{Lemma}
\newtheorem{prop}[thm]{Proposition}
\newtheorem{defn}[thm]{Definition}

\theoremstyle{remark}
\newtheorem{remark}{Remark}[section]

\numberwithin{equation}{section}

\begin{document}
\title[PKS-NS]{Global existence of free-energy solutions to the 2D Patlak--Keller--Segel--Navier--Stokes system with critical and subcritical mass}

\author[C. Lai]{Chen-Chih Lai}
\address{\noindent
Department of Mathematics,
University of British Columbia, Vancouver, B.C., V6T 1Z2, Canada}
\email{chenchih@math.ubc.ca}

\author[J. Wei]{Juncheng Wei}
\address{\noindent
Department of Mathematics,
University of British Columbia, Vancouver, B.C., V6T 1Z2, Canada}
\email{jcwei@math.ubc.ca}

\author[Y. Zhou]{Yifu Zhou}
\address{\noindent
Department of Mathematics,
Johns Hopkins University, 3400 N. Charles Street, Baltimore, MD 21218, USA}
\email{yzhou173@jhu.edu}

\begin{abstract}
We consider a coupled Patlak--Keller--Segel--Navier--Stokes system in $\R^2$ that describes the collective motion of cells and fluid flow, where the cells are attracted by a chemical substance and transported by ambient fluid velocity, and the fluid flow is forced by the friction induced by the cells. The main result of the paper is to show the global existence of free-energy solutions to the 2D Patlak--Keller--Segel--Navier--Stokes system with critical and subcritical mass.
\end{abstract}
\maketitle





\medskip

\section{Introduction}

\medskip

In this paper, we consider the Cauchy problem of the following 2D Patlak--Keller--Segel--Navier--Stokes system (PKS--NS system) \cite{GH-arxiv2020}

\EQ{\label{PKS-NS}
\begin{cases}
\partial_t n + u\cdot \nabla n =\Delta n-\nabla \cdot(n\nabla c)& \text{ in }\R^2\times(0,\infty),\\
-\Delta c= n& \text{ in }\R^2\times(0,\infty),\\
\partial_t u+u\cdot \nabla u +\nabla P=\Delta u +n\nabla c,\quad \nabla\cdot u=0& \text{ in }\R^2\times(0,\infty),\\
(n,u)(\cdot,0)=(n_0,u_0)& \text{ in }\R^2,
\end{cases}
}
where $n$, $c$, $u$ and $P$ represent the cell density, the concentration of chemoattractant, fluid velocity field and pressure, respectively, and $(n_0,u_0)$ is a given initial data such that $\na \cdot u_0 = 0$. The system $\eqref{PKS-NS}_1$-$\eqref{PKS-NS}_2$ is the parabolic-elliptic Patlak--Keller--Segel equations with additional effect of advection by a shear flow \cite{BH-SIMA2017, BH-SIMA2018, He-2018}, which models the chemotaxis phenomena in a moving fluid. The equations $\eqref{PKS-NS}_3$ is known as the forced Navier--Stokes equations describing the fluid motion subject to forcing generated by the cells. The forcing $n\na c$, which also appears in the Nernst--Planck--Navier--Stokes system \cite{CI2019} with an opposite sign, is to make the cells move without acceleration and to match the aggregation nonlinearity in the cell density evolution system. 

\medskip

The system \eqref{PKS-NS} can be viewed as a coupling between the Patlak--Keller--Segel system (PKS system) and the incompressible Navier--Stokes equations (NS equations). Many significant contributions have been made on the studies to the Patlak--Keller--Segel system and the incompressible Navier--Stokes equations. The Patlak--Keller--Segel system is first introduced by Patlak \cite{Patlak1953} and Keller and Segel \cite{KS1970} and simplified by Nanjundiah \cite{Nanjundiah1973}. We refer to \cite{Horstmann-2003,Horstmann-2004} for a review of the mathematical problems in chemotaxis models. It is well known that the PKS system is $L^1$ critical \cite{JL1992} and the $L^1$ norm of the solution is conserved. When the initial measure has a small atomic part, a unique mild solution is constructed in \cite{Biler-StudMath1995}. There is a competition between the tendency of cells to spread all over $\R^2$ by diffusion and the tendency to aggregate due to the drift caused by the chemoattractivity. The balance between these two mechanisms happens exactly at the critical mass $M=8\pi$. In fact, in the subcritical case $M<8\pi$, the global well-posedness of the free energy solution of PKS system is established in \cite{DP-2004,BDP-EJDE2006} under the assumptions of finite initial second moment and finite initial entropy. The proof relies on the free energy functional introduced in \cite{NSY-1997,Biler-AMSA1998,GZ-1998} and the logarithmic Hardy--Littlewood--Sobolev inequality. We refer the reader to \cite{CCLW-2012} for an alternative proof using Delort's theory on the 2D incompressible Euler equation \cite{Delort-1991}. For the asymptotic behavior of the solutions in the subcritical case, we refer the interested reader to \cite{BDEF-2010,CD-2014}. On the other hand, for the case of supercritical mass $M>8\pi$, under suitable assumption, the solution blow up in finite time due to the evolution equation of second moment, see \cite{Biler1994, Biler1995, Nagai1995, HV-JMB1996, HV-MathAnn1996, HV-1997, Biler-AMSA1998, Biler-AdvDE1998, Nagai2001, Velazquez-2002, BDP-EJDE2006, DS-2009, RS-2014, CGMN-arxiv2019} for example. Let us mention that the continuation of measure-valued solutions of the PKS system in the sense of Poupaud's weak solutions \cite{Poupaud-2002}, which make sense even if there are mass concentrations, have also been studied (see \cite{DS-2009,Velazquez-2004-1,Velazquez-2004-2,LSV-2012,BM-ARMA2014,BZ-2015}). As for the critical case $M=8\pi$, solutions are shown to be global in time (see \cite{Biler2006} for the radial case and \cite{Velazquez-2004-1} for the general case). 
An infinite time blow-up free energy solution having a finite second moment was found in \cite{BCM-CPAM2008}. The solution converges to a Dirac measure as time goes to infinity. If $M\le8\pi$ and the entropy of the initial data is finite, the global-in-time existence of classical solutions of finite entropy without the assumption of finite second moment is proved in \cite{NO-2016} (also see \cite{DNR-1998} in the radial case). A classical survey can be found in \cite{Blanchet-2013}. The precise description of the infinite-time dynamic and the limiting profile is given in \cite{CS-2006, LNY-2013, GM-2018}. Recently, it is shown in \cite{davila2019infinite} that there exists a radial function $n_0^*$ with mass $8\pi$ such that for any initial data $n_0$ sufficiently close to $n_0^*$ the solution starting from $n_0$ is globally defined and blows up at infinite time, with precise blow-up behaviors and stability obtained. It is shown in \cite{Wei2018} using a monotonicity formula that mild solutions of PKS exists globally in time if and only if $M\le8\pi$ without any additional assumptions. We refer readers to \cite{BCC2012} for non-blow-up free energy solutions which have an infinite second moment. For the incompressible Navier--Stokes equation, the existence of global weak solutions to the initial value problem has been established by Leray \cite{Leray1934Acta} and Hopf \cite{Hopf1951Acta}. For comprehensive results regarding the Navier--Stokes equation, we refer the interested readers to \cite{Sohr-book2001, Temam2001book, Lions1996book, RieussetNS, Galdi, Seregin, Tsai, BS-2018}, the references therein, and \cite{Kato-1984, GMO-1988, GS-1991, Kato-1994, Ben-Artzi-1994, Brezis-1994, ORS-1997, GMS-2001, GW-2002, GG-2005, GW-2005, BB-2009} for results in two and general dimensions.

\medskip


Similar to the PKS system, the PKS--NS system \eqref{PKS-NS} formally preserves mass, in the sense that
\[
\int_{\R^2} n(x,t)\, dx = \int_{\R^2} n_0(x)\, dx =:M
\]
for all $t\in(0,\infty)$ because $u$ is divergence-free. Moreover, since the solution of the two dimensional Poisson equation in $\eqref{PKS-NS}_2$ is given up to a harmonic function, we directly define the concentration of the chemoattractant by
\[
c(x,t) = \frac1{2\pi} \int_{\R^2} \log \frac1{|x-y|}\, n(y,t)\, dy.
\]
Furthermore, the PKS--NS system \eqref{PKS-NS} possesses a decreasing free energy functional \cite[(1.2)]{GH-arxiv2020}
\EQ{\label{eq-free-energy}
\mathcal{F}[n,u]:=\int_{\R^2} \bkt{ n\bke{\log n - \frac12\, c} + \frac12\, |u|^2 }\, dx,
}
where the first and the second terms are the entropy and a potential energy of the density $n$, respectively, and the third term is the kinetic energy of the velocity field $u$. As shown in \cite[Lemma 1.1]{GH-arxiv2020}, the free energy is dissipative along the dynamics \eqref{PKS-NS}.

\medskip

\begin{lem}[Free Energy Functional]\thlabel{free-energy}
Let $(n,u)$ be a smooth solution of \eqref{PKS-NS} and $\mathcal{F}[n,u]$ be given in \eqref{eq-free-energy},
then 
\EQN{
\frac{d}{dt}\, \mathcal{F}[n, u] = - \int_{\R^2} \bkt{ n |\na(\log n - c)|^2 +  |\na u|^2} dx.
}
\end{lem}
\begin{proof}
A direct computation shows
\EQN{
\frac{d}{dt}\, \mathcal{F}[n, u] = - \int_{\R^2} n |\na(\log n - c)|^2\, dx + \int_{\R^2} nu\cdot \na c\, dx -  \int_{\R^2} |\na u|^2\, dx + \int_{\R^2} n\na c\cdot u\, dx.
}
The lemma follows by integrating the last term by parts and using $\na \cdot u=0$.
\end{proof}

\medskip

In \cite{GH-arxiv2020}, the local and global existence of solutions to the Patlak--Keller--Segel--Navier--Stokes system \eqref{PKS-NS} are established in the Sobolev space $H^s$, $s\ge2$, when the initial mass is strictly less than $8\pi$. The $H^s$ bound of their solutions grows exponentially in time. They also consider \eqref{PKS-NS} on a torus $\mathbb{T}^2$, and prove a similar existence result as for $\R^2$ while the solution can be  bounded in $H^s$ uniformly in time. There are many open problems regarding the PKS--NS system \eqref{PKS-NS}. For example, it is unclear that solutions exists globally in time or there exists a finite-time singularity when $M\ge8\pi$. The ambient fluid flow might suppress the potential blow-up in $\eqref{PKS-NS}_1$-$\eqref{PKS-NS}_2$ (see \cite{KX-2016, BH-SIMA2017, BH-SIMA2018, He-2018}). Despite of the known results of suppression of explosion on torus, finite-time blow-ups of \eqref{PKS-NS} are possible in $\R^2$ when $u$ is divergence-free. 
In a forthcoming work \cite{infinitePKSNS}, we construct infinite time blow-up solutions to the PKS--NS system in the critical case $M=8\pi$ by a new inner-outer gluing method. It is worth mentioning that in the radially symmetric class, the PKS--NS system is decoupled which can be seen from the dynamics of the second moment \eqref{eq-2nd-moment-PKSNS} and PKS system itself can develop finite time singularities in the case $M>8\pi$ (see \cite{Velazquez-2002, CGMN-arxiv2019, Mizoguchi-2020} for example). We will get back to the construction of finite time blow-up in the supercritical case $M>8\pi$ for the PKS--NS system without any symmetry in a future work. In another aspect, the PKS--NS system in a bounded domain with Neumann boundary condition is also an intriguing problem.


\medskip

The main aim of this paper is to show the existence of global-in-time solutions for the Patlak--Keller--Segel--Navier--Stokes system \eqref{PKS-NS} with mass $M\le8\pi$. Throughout the paper, the initial data are assumed to satisfy
\EQ{\label{initial-n}
0\le n_0 \in L^1(\R^2),\qquad n_0\log n_0\in L^1(\R^2), \qquad n_0\log (1+|x|)\in L^1(\R^2),
}
and
\EQ{\label{initial-u}
u_0\in W^{1,2}_{0,\si}(\R^2) := \overline{C^\infty_{0,\si}(\R^2)}^{\norm{\cdot}_{W^{1,2}(\R^2)}} = \{u\in W^{1,2}(\R^2) : \div u = 0\},
}
where $C^\infty_{0,\si}(\R^2):= \{u\in C^\infty_0(\R^2) : \div u = 0\}$. 

\medskip

We now define the free-energy solution of the PKS--NS system \eqref{PKS-NS} of which the concept is introduced in \cite{BCM-CPAM2008} for PKS system.

\medskip

\begin{defn}[Free-Energy Solution] Given $T>0$, $(n,u)$ is a free-energy solution to the Patlak--Keller--Segel--Navier--Stokes system \eqref{PKS-NS} with initial data $(n_0,u_0)$ on $[0,T]$ if $(1+|\log n|)n\in L^\infty(0,T;L^1(\R^2))$, $u\in L^\infty(0,T;L^2(\R^2))\cap L^2(0,T;\dot H^1(\R^2)$, $(n,u)$ satisfies \eqref{PKS-NS} in distributional sense, and
\EQ{\label{ineq-free-energy}
\mathcal{F}[n, u](t) + \int_0^t\int_{\R^2} \bkt{ n(x,s) |\na(\log n(x,s) - c(x,s))|^2 + \frac12\, |\na u(x,s)|^2} dxds \le \mathcal{F}[n_0, u_0]
}
for almost every $t\in(0,T)$.
\end{defn}

\medskip

It will be shown in Section \ref{S-mild-solution} that every mild solution $(n,u)$ is a free-energy solution.

\medskip

The following is our main result on global-in-time existence of solutions to \eqref{PKS-NS}.

\medskip

\begin{thm}[Global existence]\thlabel{NO-thm1.2}
Suppose that $(n_0,u_0)$ satisfies the assumptions \eqref{initial-n} and \eqref{initial-u}, and
\[
M := \int_{\R^2} n_0\, dx \le 8\pi,
\]
then the free-energy solution $(n,u)$ to the Patlak--Keller--Segel--Navier--Stokes system \eqref{PKS-NS} exists globally in time. Moreover, the $L^\infty$-norm of the solution does not blow up in finite time, i.e., for any $0<t_0<T<\infty$
\EQ{\label{eq-Linf-bound}
\sup_{t\in(t_0,T)} \bke{ \norm{n(t)}_\infty + \norm{u(t)}_\infty } <\infty.
}
\end{thm}

\medskip

In order to prove \thref{NO-thm1.2}, we need to use the following proposition.

\medskip

\begin{prop}\thlabel{NO-thm3.1}
Let $(n_0,u_0)$ satisfy \eqref{initial-n} and \eqref{initial-u}. Assume also $\norm{n_0}_1 = 8\pi$. For a local-in-time mild solution $(n,u)$ on $[0,T)$ with initial data $(n_0,u_0)$ given in \thref{local-exist}, it holds that for $0<t_0<T$,
\[
\sup_{t_0\le t<T} \int_{\R^2} (1+n(t)) \log(1+n(t))\, dx < \infty.
\]
\end{prop}

\medskip

Although the PKS--NS system shares many features with the PKS system because of $u$ being divergence-free, there are still some significant properties holds for PKS but fails for PKS--NS. It is well-known that solutions of the PKS system with suitable spatial decay also satisfy the conservation of the first moment (the center of mass) $\int_{\R^2} n(x,t)\,x\, dx$, and the equation of the second moment (the variance) $\frac{d}{dt} \int_{\R^2} |x|^2 n(x,t)\, dx = 4M\bke{1-\frac{M}{8\pi}}$. From the evolution equation of the second moment, we can easily see that solutions with finite initial second moment do not exist globally if $M>8\pi$. Moreover, the second moment is conserved for the critical case $M=8\pi$, which is one of the keys to prove the global existence in \cite{BCM-CPAM2008}. However, for the PKS--NS system \eqref{PKS-NS}, unlike the PKS system, the dynamics of the first moment and the second moment are more delicate due to the coupling from the velocity field. In fact, they satisfy
\[\frac{d}{dt}\int_{\R^2} n(x,t)x_i\,dx = \int_{\R^2} n(x,t)u_i(x,t)\,dx\]
and
\EQ{\label{eq-2nd-moment-PKSNS}
\frac{d}{dt} \int_{\R^2} |x|^2 n(x,t)\,dx = 4M - \frac{M^2}{2\pi} + 2 \int_{\R^2}n(x,t) u(x,t)\cdot x\, dx.
}
The sign of the last term in \eqref{eq-2nd-moment-PKSNS} is unknown. It is unclear how to control the integral unless the solution is radially symmetric \cite[Corollary 1]{GH-arxiv2020}, in which case the integral vanishes (the system gets decoupled). The lack of control over the second moment is the main obstacle to bound solutions uniform in time. Instead of using the second moment, we employ the following modified free energy functional
\EQ{\label{eq-mfree}
H[n,u](t) := \int_{\R^2} (1+n(t)) \log(1+n(t))\, dx - \frac12 \int_{\R^2} n(t) c(t)\, dx + \frac12 \int_{\R^2} |u(t)|^2\, dx
}
and the Brezis--Merle inequality as in \cite{NO-2016} to estimate the modified free energy in the interior region in the proof of \thref{NO-thm3.1}. Details of the proof of \thref{NO-thm3.1} are given in Section \ref{S-apriori-PKS}. It is worth mentioning that the modified entropy $\int_{\R^2}(1+n(t)) \log(1+n(t))\, dx$ in \eqref{eq-mfree} is non-negative and has been used to get nonnegative global solutions for parabolic-parabolic Keller--Segel system in \cite{NO-2011} and \cite{Mizoguchi-2013}.

\medskip

\begin{remark}
Unlike the free energy functional $\mathcal{F}[n,u]$, the modified energy $H[n,u]$ defined in \eqref{eq-mfree} does not decrease in time. In fact, it formally satisfies
\EQN{
\frac{d}{dt}\, H[n,u] 
&= -\int_{\R^2} \abs{\na \log(1+n) - \frac{\na c}2}^2 dx - \int_{\R^2} n \abs{\na \log(1+n) - \na c}^2 dx - \int_{\R^2} |\na u|^2\, dx\\
&\quad + \int_{\R^2} \frac{|\na c|^2}4\, dx.
}
\end{remark}

\medskip

Let us explain the idea for the proof of our main result, \thref{NO-thm1.2}. First of all, we employ the free energy functional $\mathcal{F}[n,u]$ defined in \eqref{eq-free-energy} and the logarithmic Hardy--Littlewood--Sobolev inequality to obtain the a priori estimate of the fluid velocity. For the critical case $M=8\pi$, the energy of fluid velocity is uniformly bounded in time (see \thref{u-a-priori}), and thus the velocity of mild solution verify the Prodi--Serrin conditions for regularity (see \thref{u-regularity} and \thref{u-regularity-1}). The regularity of $u$ allows us to control all additional terms arising from the coupling when deriving the a priori entropy estimate of the advection-Patlak--Keller--Segel system $\eqref{PKS-NS}_1$-$\eqref{PKS-NS}_2$ in Section \ref{S-apriori-PKS}. For the subcritical regime $M<8\pi$, we improve the result of \cite{GH-arxiv2020} by reducing the regularity assumption on the initial data $(n_0,u_0)$. In fact, we show that solutions with the initial conditions \eqref{initial-n} and \eqref{initial-u} become $H^s$ in short time, thus the global-in-time existence follows directly from \cite{GH-arxiv2020}.

\medskip

The rest of this paper is organized as follows. In Section \ref{S-preliminary}, we prepare some analysis tools that used in this paper. In Section \ref{S-mild-solution}, we introduce the notion of mild solutions of \eqref{PKS-NS} and discuss local existence and regularity of the mild solutions. In Section \ref{S-apriori-NSE}, we obtain a priori estimates of the forced Navier--Stokes equation. In Section \ref{S-apriori-PKS}, we derive a priori estimates of the modified entropy and prove \thref{NO-thm3.1}. Section \ref{S-global-exist} is devoted to prove \thref{NO-thm1.2}, the global existence of solutions to \eqref{PKS-NS}.

\bigskip

\section{Preliminary}\label{S-preliminary}

\medskip

Before proving our main theorem, we recall some well-known and useful lemmas. 

\medskip

In the light of the critical Sobolev embedding in $\R^2$, we define the class of functions of the bounded mean oscillations $BMO(\R^2)$ by 
\[
\norm{f}_{BMO} := \sup_{B} \frac1{|B|} \int_B |f-f_B|\, dx \quad \text{ with } f_B:= \frac1B \int_B f\, dx,
\]
where the supremum being taken over the set of Euclidean balls $B$ and $|B|$ is the Lebesgue measure of $B$. It is well-known that $H^1(\R^2)$ is embedded in $BMO(\R^2)\cap L^2(\R^2)$ as a corollary of the Poicar\'{e} inequality.

\medskip

In view of the modified entropy $\int_{\R^2} (1+n) \log(1+n)\, dx$ appears in the modified free energy, we recall the notion of the Orlicz space.

\medskip

\begin{defn}[Orlicz space \cite{Tr-1967, BMM-2011}]
Let $\phi:\R_+\to\R_+$ be a convex function such that
\EQ{\label{def-orlicz}
\lim_{s\to0_+} \frac{\phi(s)}s = 0,\quad \lim_{s\to\infty} \frac{\phi(s)}s = \infty.
}
Then the Orlicz class $L_\phi(\R^2)$ consists of all measurable functions $f:\R^2\to\C$ satisfying 
\[
\int_{\R^2} \phi(|f(x)|) dx <\infty.
\]
\end{defn}

\medskip

Since $\phi(s)=(1+s)\log(1+s)-s$ is a non-negative, convex function satisfying \eqref{def-orlicz}, non-negative functions with finite total mass and finite modified entropies forms the Orlicz class $L_{(1+\cdot)\log(1+\cdot)}=:L\log L$. If $f\in L\log L$, then the solution $\psi$ to the Poisson equation 
\EQ{\label{eq-Poisson}
-\De\psi = f \quad \text{ in }\R^2
}
is a locally bounded function. If $f\log(1+|x|)\in L^1$ then $\psi$ becomes a locally integrable function.

\medskip

The following proposition of $BMO$ estimate for the solution of two dimensional Poisson equation is well known.

\medskip

\begin{prop}[$BMO$ estimate]\thlabel{BMO-est}
Let $f\log(2+|x|)\in L^1(\R^2)$ and $\psi$ be the solution of the Poisson equation \eqref{eq-Poisson}. Then we have
\[
\norm{\psi}_{BMO} \le C\norm{f}_1,
\]
where $C$ is a constant independent of $f$.
\end{prop}

\begin{proof}
It is well known that $\log|x|\in BMO$. Since
\[
\psi(x) - \frac1{|B|} \int_B \psi(y)\, dy = \int_{\R^2} \bkt{\log|x-y| - \frac1{|B|} \int_B \log|z-y|\, dz} f(y)\, dy,
\]
we have
\EQN{
\frac1{|B|} \int_B& \abs{\psi(x) - \frac1{|B|} \int_B \psi(y)\, dy} dx\\
&\le \frac1{|B|} \int_{\R^2} \int_B \abs{\log|x-y| - \frac1{|B|} \int_B \log|z-y|\, dz} dx |f(y)|\, dy\\
& = \frac1{|B|} \int_{\R^2} \int_{B_y} \abs{\log|x| - \frac1{|B_y|} \int_{B_y} \log|z|\, dz} dx |f(y)|\, dy,\quad B_y = B+y,\\
&\le \frac1{|B|}\, \norm{\log\abs{\cdot}}_{BMO} \norm{f}_1,
}
completing the proof.
\end{proof}

\medskip

$BMO$ functions are locally $L^p$. In fact, using the John--Nirenberg inequality, one can prove the following lemma.

\medskip

\begin{lem}[\hspace{-0.125cm}$\text{\cite[(7), pp. 144]{Stein-book}}$]\thlabel{BMO-def-est}
Suppose $f\in BMO(\R^2)$. Then for any $1<p<\infty$, $f$ is locally in $L^p$, and 
\[
\frac1{|B|} \int_B |f-f_B|^p\, dx \le C_p \norm{f}_{BMO}^p
\]
for any balls $B$.
\end{lem}

\medskip

We also make use of the logarithmic Hardy--Littlewood--Sobolev inequality to derive the energy estimate of the velocity field $u$.  

\medskip

\begin{prop}[Logarithmic Hardy--Littlewood--Sobolev Inequality \cite{Beckner1993, CL1992}]\thlabel{log-HLS}
Let $f$ be a nonnegative function in $L^1(\R^2)$ such that $f\log f $ and $f\log(1+|x|^2)$ belong to $L^1(\R^2)$. If $\int_{\R^2}f\, dx = M$, then
\[
\int_{\R^2} f\log f\, dx + \frac2M \int_{\R^2}\int_{\R^2} f(x) f(y) \log|x-y|\, dxdy \ge -C(M)
\]
with $C(M):= M(1+\log\pi - \log M)$.
\end{prop}

\medskip

The logarithmic Hardy--Littlewood--Sobolev inequality has been broadly applied to the PKS system. It is the key ingredient to show the global-in-time existence in \cite{BDP-EJDE2006} and to characterize the blowup profile in \cite{BCM-CPAM2008}. To be more precise, it is used to derive an upper bound of the entropy from the dissipative free energy functional for the PKS system.

\medskip

As mentioned in the introduction, we employ the modified free energy functional $H[n,u]$ given in \eqref{eq-mfree}. The following lemma implies the equivalence of the free energy function $\mathcal{F}[n,u]$ and the modified one $H[n,u]$ under the weaker condition that $n_0\log(2+|x|)\in L^1$.

\medskip

\begin{lem}[\hspace{-0.125cm}$\text{\cite[Lemma 2.3]{NO-2016}}$]\thlabel{NO-lem2.3-2016}
Let $f$ be a nonnegative measurable function on a measurable set $\Omega$ in $\R^n$ satisfying $f\log(2+|x|)\in L^1(\Omega)$. Then 
\[
f\log f \in L^1(\Omega)\qquad \Leftrightarrow \qquad (1+f)\log(1+f)\in L^1(\Omega).
\]
In fact,
\[
\int_\Omega (1+f) \log(1+f)\, dx \le 2\int_\Omega f|\log f|\, dx + (2\log2) \int_\Omega f\, dx
\]
and
\[
\int_\Omega f|\log f|\, dx \le \int_\Omega (1+f) \log(1+f)\, dx + 2\al \int_\Omega f\log(2+|x|)\, dx + \frac1e \int_\Omega \frac1{(2+|x|)^\al}\, dx,
\]
where $n<\al<\infty$.
\end{lem}

\medskip

The following lemma is the key to reach the global existence \thref{NO-thm1.2} from \thref{NO-thm3.1}. It gives us a control of the $L^3_{t,x}$ norm of $n$ in the energy equality of the density $n$.

\medskip

\begin{lem}[\hspace{-0.125cm}$\text{\cite[Lemma 2.1 (2)]{NO-2011}}$]
For any $\ve>0$, there exists $C_\ve>0$ such that for $f\ge0$
\EQ{\label{NO-2011-lem2.1-1}
\norm{f}_2 \le \ve \norm{(1+f)\log(1+f)}_1^{\frac12} \norm{\na f}_2^{\frac12} + C(\ve) \norm{f}_1^{\frac12}
}
\EQ{\label{NO-2011-lem2.1-2}
\norm{f}_3 \le \ve \norm{(1+f)\log(1+f)}_1^{\frac13} \norm{\na f}_2^{\frac23} + C(\ve) \norm{f}_1^{\frac13}
}
where $C(\ve)=O((e^{\ve^{-3}}-1)^{\frac23})$ as $\ve\to0$.
\end{lem}

\medskip

In order to control the free energy in the interior region, we need the following lemma concerning the Poisson equation in $\R^2$, which is a consequence of the Brezis--Merle inequality under the zero Dirichlet boundary condition.

\medskip

\begin{lem}[\hspace{-0.125cm}$\text{\cite[Lemma 2.7]{NO-2016}}$]\thlabel{NO-lem2.7}
Let $\Omega$ be a bounded domain in $\R^2$ with smooth boundary. For $g\in L^2(\Omega)$, let $v\in W^{2,2}(\Omega)$ be a solution of $-\De v = g$ in $\Omega$. If $\norm{g}_{L^1(\Omega)}<4\pi$, then
\[
\int_\Omega \exp\bke{|v(x)|}\, dx \le \frac{4\pi^2}{4\pi - \norm{g}_{L^1(\Omega)}}\, \textup{d}(\Omega)^2 \exp\bke{\sup_{\pd\Omega}|v(x)|}
\]
where $\textup{d}(\Omega)$ is the diameter of $\Omega$.
\end{lem}

\medskip

It follows from the following lemma that the $L^\infty$-norm of $\na c$ is controlled by the $L^q$-norm of $n$, $q\in(2,\infty)$.

\medskip

\begin{lem}[\hspace{-0.125cm}$\text{\cite[Lemma 2.5]{Nagai-2011}}$]\thlabel{Nagai2011-lem2.5}
Let $2<q\le\infty$. For $f\in L^1\cap L^q$,
\[
\norm{\na \frac1{2\pi}\, \log \frac1{|\cdot|} * f}_\infty \le C_q \norm{f}_1^{\frac{q-2}{2(q-1)}} \norm{f}_q^{\frac{q}{2(q-1)}},
\]
where
\[
C_q = (2\pi)^{1/2} \bke{\frac{q-1}{q-2}}^{1/2} \bket{\bke{\frac{q}{q-1}}^{\frac{q-2}{2(q-1)}} + \bke{\frac{q}{q-1}}^{-\frac{q}{2(q-1)}}}.
\]
\end{lem}

\medskip

We denote for any real-valued function $f=f(x)$ that $f^+(x) = \max\{f(x),0\}$ and $f^-(x) = (-f)^+(x)$, so that $f=f^*-f^-$. In the same spirit as \cite[Lemma 2.2]{BCM-CPAM2008} where the negative part of the entropy is controlled by the second moment, we can control the negative part of the entropy by the weighted $L^1$ norm with the weight $(1+\log(1+|x|^2))$.

\medskip

\begin{lem}[Control of the negative part of the entropy]\thlabel{lem-neg-entropy}
For any $g$ such that $(1+\log(1+|x|^2)) g \in L^1_+(\R^2)$, we have $g\log g^- \in L^1(\R^2)$ and
\[
\int_{\R^2} g(x) \log^-g(x)\, dx
\le 2 \int_{\R^2} g(x) \log(1+|x|^2)\, dx + \log \pi \int_{\R^2} g(x)\, dx + \frac1e.
\]
\end{lem}

\begin{proof}
Let $u:= g \mathbbm{1}_{\{g\le1\}}$ and $m := \int_{\R^2} u\, dx \le M:=\int_{\R^2} g\, dx$. Then
\[
\int_{\R^2} u\bkt{\log u + 2\log(1+|x|^2)} dx
= \int_{\R^2} U \log U \mu\, dx - m \log \pi,
\]
where $U:= u/\mu$ and $\mu(x) = \frac1{\pi(1+|x|^2)^2}$ so that $\int_{\R^2} \mu\, dx =1$. By Jensen's inequality,
\[
\int_{\R^2} U\log U \mu\, dx \ge \bke{\int_{\R^2} U\, \mu dx} \log \bke{\int_{\R^2} U \mu\, dx} 
= m\log m 
\ge - \frac1e.
\]
So
\[
\int_{\R^2} u\bkt{\log u + 2\log(1+|x|^2)} dx
\ge - \frac1e - m\log \pi.
\]
Therefore,
\[
\int_{\R^2} g \log^- g\, dx 
= - \int_{\R^2} u \log u\, dx
\le 2 \int_{\R^2} u \log(1+|x|^2)\, dx + m \log \pi + \frac1e,
\]
completing the proof since $u\le g$ and $m\le M$.
\end{proof}

\bigskip

\section{Local-in-time existence of mild solutions}\label{S-mild-solution}

\medskip

We begin with the definition of the mild solution of the PKS--NS system \eqref{PKS-NS}.

\medskip

\begin{defn}[Mild Solution]
Given $(n_0,u_0) \in L^1\times L^2_{\si}$, we define $(n(t),u(t))$ to be a mild solution of the Patlak--Keller--Segel--Navier--Stokes system \eqref{PKS-NS} on $[0,T)$ with initial data $(n_0,u_0)$ if
\EN{
\item [(i)] $n \in C([0,T);L^1) \cap C((0,T); L^{4/3})$ and $u \in C([0,T);L^2) \cap C((0,T); L^4)$,
\item [(ii)] $\underset{0<t<T}{\sup} t^{1/4} \bke{\norm{n(t)}_{4/3} + \norm{u(t)}_4}< \infty$,
\item [(iii)] $(n,u)$ satisfies the integral equations
\EQ{\label{Duhamel-n}
n(t) = e^{t\De} n_0 - \int_0^t e^{(t-s)\De} \na \cdot (n(s)\na c(s) + n(s)u(s))\, ds,
}
\EQ{\label{Duhamel-u}
u(t) = e^{t\De} u_0 - \int_0^t e^{(t-s)\De} {\bf P}\na \cdot (\na c(s)\otimes \na c(s) + u(s) \otimes u(s))\, ds,
}
where $-\De c(s) = n(s)$.
}
\end{defn}

\medskip

In this section, we discuss the local existence and regularity of mild solutions to \eqref{PKS-NS}.

\medskip

Let us write \eqref{Duhamel-n}-\eqref{Duhamel-u} as
\[
(n(t),u(t)) = (e^{t\De} n_0,e^{t\De} u_0) - B((n,u),(n,u)),
\]
where $B=(B_1,B_2)$ is a bilinear form in which $B_1$ and $B_2$ are bilinear forms defined by
\[
B_1((m,v),(n,u)) := \int_0^t e^{(t-s)\De} \na \cdot  (m(s)\na c(s) + m(s) u(s))\, ds,
\]
\[
B_2((m,v),(n,u)) := \int_0^t e^{(t-s)\De} {\bf P} \na \cdot (\na b(s) \otimes \na c(s) + v(s) \otimes u(s))\, ds
\]
in which $-\De c(s) = n(s)$ and $-\De b(s) = m(s)$.

\medskip

By the Hardy--Littlewood--Sobolev inequality, for any $q\in(2,\infty)$
\EQ{\label{HLS}
\norm{\na c}_{q}\lec\norm{n}_{\frac{2q}{2+q}}.
}
We recall the classical $L^q$-$L^p$ estimates of heat and Stokes semigroups. For any $1\le q \le p \le \infty$ $(p\neq 1, q\neq\infty)$, there holds
\EQ{\label{semigp-est}
\norm{e^{t\De} f}_p + \norm{e^{t\De} {\bf P} f}_p &\lec t^{\frac1p - \frac1q} \norm{f}_q,\\
\norm{e^{t\De} \na\cdot F}_p + \norm{e^{t\De} {\bf P} \na\cdot F}_p &\lec t^{-\frac12 + \frac1p - \frac1q} \norm{F}_q.
}

\medskip

Let $\calE_T$ be a Banach space defined as
\EQN{
\calE_T = 
&\left\{
(n,u)\in L^\infty(0,T; L^1)\times L^\infty(0,T; L^2):\right.\\ 
&\qquad \left.t^{1/4}(n(\cdot,t), u(\cdot,t)) \in L^\infty(0,T; L^{4/3})\times L^\infty(0,T; L^{4})
\right\}
}
with the norm
\EQN{
\norm{(n,u)}_{\calE_T}:=& 
\sup_{t\in(0,T)} t^{1/4}\bke{\norm{n(\cdot,t)}_{4/3} + \norm{u(\cdot,t)}_{4}}
+
\sup_{t\in(0,T)}\bke{\norm{n(\cdot,t)}_{1} + \norm{u(\cdot,t)}_{2}}
\\
:=& \norm{(n,u)}_{X_T} + \norm{(n,u)}_{Y_T}.
}

\medskip

We first establish the local-in-time existence of mild solutions to \eqref{PKS-NS} and some important properties of the solutions in the following theorem.

\medskip

\begin{thm}\thlabel{local-exist}
Given $n_0\in L^1$ and $u_0\in L^2_\si$, there exists $T\in(0,\infty)$ such that the Cauchy problem \eqref{PKS-NS} has a unique mild solution $(n,u)$ on $[0,T)$. Moreover, $(n,u)$ satisfies the following properties:
\EN{
\item [(i)] $n(t)\to n_0$ in $L^1$ and $u(t)\to u_0$ in $L^2$;
\item [(ii)] for every $1\le p\le \infty$, there holds $n\in C((0,T];L^p)$ and $\sup_{0<t<T}(t^{1-1/p}\norm{n(t)}_p)<\infty$;
\item [(iii)] for every $1\le p\le 2$, there holds $u\in C((0,T];L^{\frac{2p}{2-p}})$ and $\sup_{0<t<T}(t^{1-1/p}\norm{u(t)}_{\frac{2p}{2-p}})<\infty$;
\item [(iv)] for every $m\in\mathbb{Z}_+$, $l\in\mathbb{Z}_+^2$ and $1<p<\infty$, there holds $\pd_t^m\pd_x^l n, \pd_t^m\pd_x^l u \in C((0,T];L^p)$;
\item [(v)] $(n,u)$ is a classical solution of \eqref{PKS-NS} in $\R^2\times(0,T)$;
\item [(vi)] If $n_0\ge0$ and $n_0\neq0$, then $n(x,t)>0$ for all $(x,t)\in\R^2\times(0,T)$;
\item [(vii)] If $n_0\log(1+|x|)\in L^1$, then $n(t)\log(1+|x|)\in L^1$ for all $0<t<T$.
}
\end{thm}
\begin{proof}
By \eqref{semigp-est} and \eqref{HLS}, we have
\EQN{
\norm{B_1((m,v),(n,u))}_{4/3} 
&~ \lec \int_0^t (t-s)^{-\frac12+\frac34-1} \norm{m(s)\na c(s) + m(s)u(s)}_{1}\, ds\\
&~ \lec \int_0^t (t-s)^{-\frac34} \norm{m(s)}_{4/3}\norm{\na c(s) + u(s)}_{4}\, ds\\
&~ \lec \bke{\int_0^t (t-s)^{-\frac34} s^{-\frac14} s^{-\frac14}\, ds} \norm{(m,v)}_{X_T} \norm{(n,u)}_{X_T}\\
&~ \lec t^{-1/4} \norm{(m,v)}_{X_T} \norm{(n,u)}_{X_T}
}
and
\EQN{
\norm{B_2((m,v),(n,u))}_{4} 
&~ \lec \int_0^t (t-s)^{-\frac12+\frac14-\frac12} \norm{\na b(s)\otimes \na c(s) + v(s) \otimes u(s)}_{2}\, ds\\
&~ \lec \int_0^t (t-s)^{-\frac34} \bke{\norm{\na b(s)}_{4}\norm{\na c(s)}_{4} + \norm{v(s)}_{4}\norm{u(s)}_{4}} ds\\
&~ \lec \int_0^t (t-s)^{-\frac34} \bke{\norm{m(s)}_{4/3}\norm{n(s)}_{4/3} + \norm{v(s)}_{4}\norm{u(s)}_{4}} ds\\
&~ \lec \bke{\int_0^t (t-s)^{-\frac34} s^{-\frac14} s^{-\frac14}\, ds} \norm{(m,v)}_{X_T} \norm{(n,u)}_{X_T}\\
&~ \lec t^{-1/4} \norm{(m,v)}_{X_T} \norm{(n,u)}_{X_T}.
}
Thus, we have proved that
\EQ{\label{local-exist-bilinear}
\norm{B((m,v),(n,u))}_{X_T} \lec \norm{(m,v)}_{X_T} \norm{(n,u)}_{X_T}.
}
We claim that if $n_0\in L^1$ and $u_0\in L^2$, then
\[
\lim_{t\to0} t^{1/4} \bke{\norm{e^{t \De} n_0}_{4/3} + \norm{e^{t \De} u_0}_{4}}= 0.
\]
The limit holds if $n_0$ and $u_0$ are smooth and have a compact support. Therefore, the claim holds by a density argument. With this claim and the bilinear estimate \eqref{local-exist-bilinear}, the local existence of a solution $(n,u)\in X_T$ of \eqref{PKS-NS} follows from the Picard iteration. The solution $(n,u)$ is in $Y_T$, and hence a mild solution of \eqref{PKS-NS}, since
\[
\norm{n(t)}_1 = \int_{\R^2} n(x,t)\, dx = \int_{\R^2} e^{t\De} n_0(x)\, dx = \int_{\R^2} n_0(x)\, dx
\]
and 
\[
\norm{u(t)}_2 \le \norm{e^{t\De} u_0}_2 + \norm{B_2((n,u),(n,u))}_2 \lec \norm{u_0}_2 + \norm{(n,u)}_{X_T}^2.
\]

It is easy to see
\[
\norm{n(t) - n_0 }_1 \le \norm{e^{t\De} n_0 - n_0}_1 + \norm{\int_0^t e^{(t-s)\De}\na\cdot(n(s)\na c(s) + n(s) u(s))\, ds}_1 \to 0
\]
and
\[
\norm{u(t) - u_0 }_2 \le \norm{e^{t\De} u_0 - u_0}_2 + \norm{\int_0^t e^{(t-s)\De} {\bf P}\na\cdot(\na c(s)\otimes \na c(s) + u(s)\otimes u(s))\, ds}_1 \to 0,
\]
as $t\to0_+$, which proves (i).

From \eqref{Duhamel-n} and \eqref{Duhamel-u}, using \eqref{HLS}, we have for $1\le q\le p\le\infty$ $(p\neq1, q\neq\infty)$ that
\EQ{\label{local-exist-ii-n}
\norm{n(t)}_p 
&\lec \norm{e^{t\De} n_0}_p + \int_0^t (t-s)^{-\frac12+\frac1p-\frac1q} \norm{n(s)\na c(s) + n(s)u(s)}_q\, ds\\
&\lec t^{\frac1p-1} \norm{n_0}_1 + \int_0^t (t-s)^{\frac1p-\frac2a} \bke{\norm{n(s)}_a^2 + \norm{n(s)}_a \norm{u(s)}_{\frac{2a}{2-a}}}\, ds,
}
and for $1\le r\le \frac{2p}{2-p} \le\infty$ $(\frac{2p}{2-p} \neq1, r\neq\infty)$ that
\EQ{\label{local-exist-ii-u}
\norm{u(t)}_{\frac{2p}{2-p}}
&\lec \norm{e^{t\De} u_0}_{\frac{2p}{2-p}} + \int_0^t (t-s)^{-\frac12+\frac{2-p}{2p}-\frac1r} \norm{\na c(s)\otimes \na c(s) + u(s)\otimes u(s)}_r\, ds\\
&\lec t^{\frac1p-1} \norm{n_0}_2 + \int_0^t (t-s)^{\frac1p-\frac2a} \bke{\norm{n(s)}_a^2 + \norm{u(s)}_{\frac{2a}{2-a}}^2}\, ds,
}
where $a=\frac{4q}{2+q}=\frac{2r}{1+r}<2$ and $\frac1p-\frac2a>-1$. 

For $1\le p<2$, we have $1\le\frac{2p}{2-p}<\infty$. Taking $a=\frac43$ so that $q=1$, $r=2$, and $\frac1p-\frac2a>-1$, \eqref{local-exist-ii-n} and \eqref{local-exist-ii-u} become
\EQ{\label{local-exist-ii-1}
\norm{n(t)}_p + \norm{u(t)}_{\frac{2p}{2-p}} 
\lec t^{\frac1p-1} + \int_0^t (t-s)^{\frac1p-\frac32} s^{-\frac12}\, ds
\lec t^{\frac1p-1}\qquad\qquad \forall\, 1\le p<2.
}

For $p=2$, we choose $a\in(\frac43,2)$ in \eqref{local-exist-ii-n} and \eqref{local-exist-ii-u}, say $a=\frac53$ so that $q=\frac{10}7$, $r=5$, and $\frac1p-\frac2q=-\frac7{10}>-1$, and get
\EQN{
\norm{n(t)}_2 + \norm{u(t)}_{\infty} 
\lec t^{-\frac12} + \int_0^t (t-s)^{-\frac7{10}} s^{-\frac45}\, ds
\lec t^{-\frac12},
}
where we used \eqref{local-exist-ii-1} with $p=\frac53$. This completes the proof of (iii). 

For $2<p<\infty$, we choose $a=\frac{4p}{1+2p}\in(\frac85,2)$ in \eqref{local-exist-ii-n} so that $q=\frac{2p}{1+p}\le p$, and $\frac1p-\frac2a=\frac{1-2p}{2p}>-1$, and get
\EQN{
\norm{n(t)}_p 
\lec t^{\frac1p-1} + \int_0^t (t-s)^{-1+\frac1{2p}} s^{\frac1{2p}-1 }\, ds
\lec t^{\frac1p-1}\qquad\qquad \forall\, 2<p<\infty,
}
where we used \eqref{local-exist-ii-1} with $p\in(\frac85,2)$.

For $p=\infty$, from \thref{Nagai2011-lem2.5}, $\norm{\na c}_\infty \lec \norm{n}_1^{\frac14} \norm{n}_3^{\frac34}$. Therefore,
\EQN{
\norm{n(t)}_\infty 
&\lec t^{-1} \norm{n(t/2)}_1 + \int_{t/2}^t (t-s)^{-\frac12 - \frac13} \bke{\norm{\na c(s)}_\infty + \norm{u(s)}_\infty} \norm{n(s)}_3\, ds\\
&\lec t^{-1} \norm{n_0}_1 + \int_{t/2}^t (t-s)^{-\frac56} s^{-\frac76}\, ds\\
&\lec t^{-1},
}
proving (ii).

We omit the proof of (iv) and (v) as they follow from the technique used in the proof of \cite[Proposition 2.3]{Nagai-2011}. The positivity of mild solutions (vi) is a direct consequence of the strong maximum principle. In fact, the proof is identical to that of \cite[Proposition 2.7]{Nagai-2011} for PKS system in which the $n$-equation is tested against $n^-$, the negative part of the density $n$. In our case, when we test $\eqref{PKS-NS}_1$ against $n^-$, the additional term vanishes because $\int_{\R^2} u\cdot\na n\cdot n^-\, dx = \int_{\R^2} u\cdot\na n^-\cdot n^-\, dx=0$. We skip the full details and refer the interested reader to the proof of \cite[Proposition 2.7]{Nagai-2011}.

To show (vii), note that 
\[
|\na \log(1+|x|^2)| \le \frac{2|x|}{1+|x|^2} \le 1 \qquad \text{ and } \qquad
|\De \log(1+|x|^2)| = \frac4{(1+|x|^2)^2} \le 4.
\]
So
\EQN{
\frac{d}{dt} & \int_{\R^2} n \log(1+|x|^2)\, dx \\
& = - \int_{\R^2} u\cdot\na n\log(1+|x|^2)\, dx + \int_{\R^2} \De n \log(1+|x|^2)\, dx - \int_{\R^2} \na \cdot(n\na c) \log(1+|x|^2)\, dx\\
& = \int_{\R^2} n u\cdot\na\log(1+|x|^2)\, dx + \int_{\R^2} n \De \log(1+|x|^2)\, dx + \int_{\R^2} n\na c\cdot\na \log(1+|x|^2)\, dx\\
& \le \int_{\R^2} n|u|\, dx + 4M + \int_{\R^2} n|\na c| dx\\
& \le \norm{n}_{4/3} \norm{u}_4 + 4M + \norm{n}_{4/3} \norm{\na c}_4\\
& \le 4M + C(\norm{n_0}_1,\norm{u_0}_2)\, t^{-\frac12}.
}
Therefore
\EQ{\label{est-nlogx}
\int_{\R^2} n (t)\log(1+|x|^2)\, dx \le \int_{\R^2} n_0 \log(1+|x|^2)\, dx + 4Mt + C(\norm{n_0}_1,\norm{u_0}_2)\, t^{\frac12}.
}
\end{proof}

\medskip

\begin{remark}\thlabel{Nagai-2011-rmk2.1}
It follows the same argument as in \cite[Remark 2.1]{Nagai-2011} with the estimate of bilinear forms $B_1$ and $B_2$ that if $n_0\in L^1\cap L^p$ and $u_0\in L^2\cap L^{\frac{2p}{2-p}}$ for $4/3\le p<2$, then $n\in BC([0,T);L^1\cap L^p)$ and $u\in BC([0,T);L^2\cap L^{\frac{2p}{2-p}})$.
\end{remark}

\medskip

\begin{remark}
By \thref{free-energy}, the mild solution constructed in \thref{local-exist} is a free-energy solution because it is a classical solution.
\end{remark}

\bigskip

\section{A priori estimates for the forced Navier--Stokes equation}\label{S-apriori-NSE}

\medskip

In order to have a better control of the modified free energy $H[n,u]$ in Section \ref{S-apriori-PKS}, we derive the following energy estimate of $u$ using the free energy $\mathcal{F}[n,u]$ and the logarithmic Hardy--Littlewood--Sobolev inequality.

\medskip

\begin{lem}[A priori bound of $u$]\thlabel{u-a-priori}
Suppose that $(n_0,u_0)$ with $n_0$ satisfying \eqref{initial-n} and $u_0\in L^2_{\si}$, and $\norm{n_0}_1\le 8\pi$. Let $(n,u)$ be a mild solution on $[0,T)$ given in \thref{local-exist}. Then
\[
\frac12\, \norm{u(t)}_2^2 + \int_0^t \norm{\na u(\tau)}_2^2\, d\tau
\le C(\norm{n_0}_1, \norm{n_0\log n_0}_1, \norm{u_0}_2,T).
\]
The constant $C$ is independent of $T$ if $\norm{n_0}_1= 8\pi$.
\end{lem}
\begin{proof}
By \thref{free-energy},
\[
\mathcal{F}[n,u](t) = \mathcal{F}[n_0,u_0] - \int_0^t \int_{\R^2} \bkt{n|\na(\log n - c)|^2 + |\na u|^2} dx.
\]
By \thref{log-HLS}, 
\[
\mathcal{F}[n_0,u_0] 
\ge \bke{1 - \frac{M}{8\pi}} \int_{\R^2} n(t) \log n(t)\, dx - \frac{M}{8\pi}\, C(M) + \frac12\, \norm{u(t)}_2^2 + \int_0^t \norm{\na u(\tau)}_2^2\, d\tau,
\]
or
\[
\frac12\, \norm{u(t)}_2^2 + \int_0^t \norm{\na u(\tau)}_2^2\, d\tau 
\le \mathcal{F}[n_0,u_0] - \bke{1 - \frac{M}{8\pi}}\int_{\R^2} n(x) \log n(x)\, dx + \frac{M}{8\pi}\, C(M).
\]
Our goals is to bound $\int_{\R^2} n\log n$ from below.
In fact, using \thref{lem-neg-entropy}
\EQN{
\int_{\R^2} n \log n\, dx 
&\ge - \int_{\R^2} n \log n^-\, dx\\
&\ge -2 \int_{\R^2} n \log(1+|x|^2)\, dx - M \log \pi - \frac1e\\
&\ge - 2 \int_{\R^2} n_0 \log(1+|x|^2)\, dx - 8 Mt - 2 C(\norm{n_0}_1,\norm{u_0}_2)\, t^{\frac12} - M \log \pi - \frac1e
}
where we used \eqref{est-nlogx} in the last inequality. We conclude that
\EQN{
\frac12\, &\norm{u(t)}_2^2 + \int_0^t \norm{\na u(\tau)}_2^2\, d\tau\\ 
& \le \mathcal{F}[n_0,u_0] - \bke{1 - \frac{M}{8\pi}}\int_{\R^2} n(x) \log n(x)\, dx + \frac{M}{8\pi}\, C(M)\\
& \le \mathcal{F}[n_0,u_0] + \bke{1 - \frac{M}{8\pi}} \bkt{ 2 \norm{n_0\log(1+|x|^2)}_1 + 8 MT + 2 C(\norm{n_0}_1,\norm{u_0}_2)\,T^{\frac12} + M \log \pi + \frac1e},
}
completing the proof.
\end{proof}

\medskip

\begin{remark}\thlabel{u-regularity}
By \cite[Lemma V.1.2.1 b)]{Sohr-book2001}
\EQN{
\bke{ \int_0^T \norm{u(\tau)}_q^s\, d\tau}^{1/s}
&\lec \frac12\,\bke{\underset{t\in[0,T)}{\text{ess\,sup}}\, \norm{u(t)}_2}^2 + \int_0^T \norm{\na u(\tau)}_2^2\, d\tau
}
for $2\le q<\infty$, $1\le s\le\infty$ satisfying $\frac2q + \frac2s = 1$. Therefore, \thref{u-a-priori} gives
\[
u\in L^s(0,T; L^q(\R^2)), \quad 2\le q<\infty, \quad 1\le s\le\infty, \quad \frac2q + \frac2s = 1
\]
with 
\[
\norm{u}_{L^s(0,T; L^q(\R^2))} \le C(\norm{n_0}_1, \norm{n_0\log n_0}_1, \norm{u_0}_2,T)
\]
in which the constant $C$ is independent of $T$ if $\norm{n_0}_1= 8\pi$.
\end{remark}

\medskip

\begin{remark}\thlabel{u-regularity-1}
If, additionally, $u_0\in W^{1,2}_{0,\si}(\R^2)$ then $\pd_t u$, $u\cdot\na u$ and $\na P$ are in $L^2(\R^2\times(0,T))$ by \cite[Theorem V.1.8.1]{Sohr-book2001}. Note that \cite[Theorem V.1.8.1]{Sohr-book2001} holds for the case $f=\na\cdot F$, $F\in L^4(\R^2\times(0,T))$. 
\end{remark}

\bigskip

\section{A priori entropy estimates for the advection-Patlak--Keller--Segel system}\label{S-apriori-PKS}

\medskip

This section is devoted to prove \thref{NO-thm3.1}, following a similar approach as in \cite{NO-2016} for PKS system. The key to derive the global-in-time a priori entropy estimate of $\eqref{PKS-NS}_1$-$\eqref{PKS-NS}_2$ is to split $\R^2$ into a ball (interior region), the complement of a larger ball (exterior region), and the annulus that connects the two regions. The main difficulties here for the PKS-NS system come from the coupling of velocity $u$ which is not present in \cite{NO-2016}, and extra efforts are needed to control the terms generated from $u$. We will use the regularity estimates of Navier-Stokes equation with forcing.

\medskip

\subsection{Entropy estimates for exterior regions}

\medskip

In this subsection, we obtain the a priori estimate for large $|x|$. To this end, we 
set 
\[
H_{\text{ext}}(t;R) := \int_{|x|\ge R} \bkt{(1+n(t))\log(1+n(t)) - n(t)} dx,
\]
and introduce the Littlewood--Paley partition of unity.

\medskip

\begin{defn}
Let $\phi(r)$ be a function in $C^\infty_0(0,\infty)$ such that $\phi(r)\in[0,1]$ and $\textup{supp}\,\phi\subset(1/2,2)$.
Then a family of the cut off functions $\{\Phi_k\}_{k=0}^\infty$ is called as the Littlewood--Paley partition of the unity in $\R^2$ if they satisfy 
\[
\Phi_k(x) = \phi(2^{-k}|x|),\qquad x\in \R^2, k=0,1,2,\ldots
\]
and $\sum_{k=0}^\infty \Phi_k\in C^\infty(0,\infty)$ with $\sum_{k=0}^\infty \Phi_k(x)=1$ for $|x|>1$ and $\textup{supp}\bke{\sum_{k=0}^\infty \Phi_k}\subset (1/2,\infty)$.
\end{defn}

\medskip

The Littlewood--Paley partition of unity $\{\Phi_k\}_{k=0}^\infty$ satisfies $\Phi_k\in C^\infty_0(\R^2)$, $\Phi_k(x)\in[0,1]$, 
\[
\textup{supp}\,\Phi_k\subset A_k := \{x\in\R^2: 2^{k-1}\le |x|\le 2^{k+1}\},
\] 
and
\[
|\na \Phi_k| \le C2^{-k}\Phi_k^{5/6},\qquad 
|\na \Phi_k^{1/2}|\le C2^{-k}\Phi_k^{1/3},\qquad
|\na^2\Phi_k|\le C2^{-2k}\Phi_k^{2/3},
\]
where $C>0$ is a constant independent of $k$.

\medskip

\begin{prop}\thlabel{NO-prop3.2}
There exists a sufficiently large $R_0$ depending only on $T$, $\norm{n_0}_1$, and $\norm{u_0}_2$ such that
\EQ{\label{NO-eq3.3}
\sup_{0\le t<T} H_{\text{ext}}(t;R_0) &+ \int_0^T \int_{|x|\ge R_0} \frac{|\na n(t)|^2}{1+n(t)}\, dxdt \\
&\qquad \le 2 H_{\text{ext}}(t;R_0/2) + C(\norm{n_0}_1, \norm{n_0\log n_0}_1, \norm{u_0}_2)T,
}
\EQ{\label{NO-eq3.4}
\int_0^T \int_{|x|\ge R_0} n^2(t)\, dxdt 
\le C \norm{n_0}_1 H_{\text{ext}}(t;R_0/2) + C(\norm{n_0}_1, \norm{n_0\log n_0}_1, \norm{u_0}_2)T.
}
\end{prop}
\begin{proof}
For fixed $k$, we let $c_m(t) = c(t) - [c(t)]_B$, where $[c(t)]_B = (1/|B|) \int_B c(t)\,dx$, $B=B_{2^{k+2}}(0)$. The equation exchanging $c(t)$ into $c_m(t)$ also holds, that is, $-\De c_m = n$. Note that by the equation of $n$,
\[
\pd_t \bket{\bkt{(1+n)\log(1+n)-n}\Phi_k} = 
- u\cdot\na n \log(1+n) \Phi_k
+ (\De n)\log(1+n)\Phi_k - \na\cdot (n\na c)\log(1+n) \Phi_k.
\]
Since
\EQN{
(\De n)\log(1+n)\Phi_k 
&= 
\nabla\cdot\bkt{\nabla n \log(1+n)\Phi_k 
- ((1+n)\log(1+n)-n)\nabla \Phi_k} \\
&\quad + ((1+n)\log(1+n)-n)\Delta \Phi_k 
- \frac{|\nabla n|^2}{1+n}\Phi_k
}
and
\EQN{
-\nabla\cdot(n\nabla c_m)\log(1+n) \Phi_k
&=
- \na\cdot\bkt{n\log(1+n)(\na c_m)\Phi_k}
+ n\log(1+n)(\na c_m)\cdot \na\Phi_k\\
&\quad + n\na\log(1+n)\cdot(\na c_m)\Phi_k,
}
we have
\EQN{
\pd_t &\bket{\bkt{(1+n)\log(1+n)-n}\Phi_k} + \frac{|\na n|^2}{1+n}\, \Phi_k\\
& = - u\cdot\na n \log(1+n) \Phi_k \\
& \quad + \na\cdot\bket{\na n\log(1+n)\Phi_k - \bkt{(1+n)\log(1+n) - n}\na\Phi_k - n\log(1+n)(\na c_m)\Phi_k}\\
& \quad + \bkt{(1+n)\log(1+n) - n}\De\Phi_k + n\log(1+n)\na c_m\cdot\na\Phi_k\\
& \quad + n\na\log(1+n)\cdot(\na c_m)\Phi_k.
}
We then have
\EQN{
\frac{d}{dt} \int_{\R^2} \bket{\bkt{(1+n)\log(1+n) - n}\Phi_k}& dx + \int_{\R^2} \frac{|\na n|^2}{1+n}\, \Phi_k\, dx\\
& = - \int_{\R^2} u\cdot\na n \log(1+n) \Phi_k\, dx + I + II + III + IV,
}
where
\EQN{
| I | &+ | II | + | III | + | IV |\\
& \le \bke{9\ve + C\int_{A_k} n(t)\, dx} \int_{\R^2} \frac{|\na n(t)|^2}{1+n(t)}\,\Phi_k\, dx\\
& \quad + C2^{-2n}(\norm{n_0}_1 + \norm{n_0}_1^2) + C(1+\ve)^2 \int_{\R^2} n(t)\Phi_k\, dx\\
& \quad + C \ve^{-5} 2^{-4n} \norm{c(t)}_{\text{BMO}}^6 + C2^{-4n} \norm{c(t)}_{\text{BMO}}^3 + C2^{-4n},
}
where $\ve$ is an arbitrary number with $0<\ve<1$ (see \cite[(3.20)]{NO-2016}). Now, using the fact that $\na\cdot u=0$ and the similar computation for $IV$ on page 87 of \cite{NO-2016}, we get
\EQN{
- \int_{\R^2} u\cdot\na n& \log(1+n) \Phi_k\, dx\\
=& - \int_{\R^2} u\cdot\na\bkt{(1+n)\log(1+n) - n}\Phi_k\, dx\\
=& \int_{\R^2} u\cdot\na\Phi_k \bkt{(1+n)\log(1+n) - n} dx\\
\le&~ C2^{-k} \int_{A_k} |u| \bkt{(1+n)\log(1+n) - n} \Phi_k^{5/6}\, dx\\
\le&~ C2^{-k} \bke{\int_{A_k} \bkt{(1+n)\log(1+n) - n}^{3/2} \Phi_k^{5/4}\, dx}^{2/3} \bke{\int_{A_k} |u|^3\, dx}^{1/3}\\
\le&~ \int_{A_k} \bkt{(1+n)\log(1+n) - n}^{3/2} \Phi_k\, dx + C2^{-3k}\norm{u}_3^3\\
\le&~ \bke{\int_{A_k}n(t)\,dx} \bke{\int_{\R^2} \frac{|\na n(t)|^2}{1+n(t)}\,\Phi_k\, dx} \\
&~+ C\ve 2^{-2k}\bke{\int_{\R^2}n_0\,dx}^2 + C\ve(1+\ve)\int_{\R^2}n(t)\Phi_k\,dx + C2^{-3k}\norm{u}_3^3.
}
By the same argument as in the proof of \cite[Proposition 3.2]{NO-2016} and taking $\ve = 1/36$, we can choose $N$ sufficiently large such that for all $k\ge N$,
\[
9\ve + (C+1) \sup_{0\le t<T} \int_{A_k} n(t)\, dx \le \frac12,
\]
and thus
\EQN{
&\frac{d}{dt} \int_{\R^2} \bket{\bkt{(1+n)\log(1+n) - n}\Phi_k} dx + \frac12 \int_{\R^2} \frac{|\na n|^2}{1+n}\, \Phi_k\, dx\\
&\le C(1+\ve)^2 \int_{\R^2} n(t)\Phi_k\, dx + C2^{-2k}\bke{\norm{c(t)}_{\text{BMO}}^6 + \norm{c(t)}_{\text{BMO}}^3 + \norm{c(t)}_{\text{BMO}}^2 + \norm{c(t)}_{\text{BMO}} + 1}\\
&\quad + C2^{-3k}\norm{u}_3^3\\
&\le C(1+\ve)^2 \int_{\R^2} n(t)\Phi_k\, dx + C2^{-2k}\bke{\norm{n_0}_1^6 + \norm{n_0}_1^3 + \norm{n_0}_1^2 + \norm{n_0}_1 + 1} + C2^{-3k} \norm{u}_3^3
}
since $\norm{c(t)}_{\text{BMO}} \lec \norm{n(t)}_1 = \norm{n_0}_1$ by \thref{BMO-est}. Thus, we have
\EQ{\label{NO-eq3.23}
\int_{\R^2} &\bkt{(1+n(t))\log(1+n(t)) - n(t)}\Phi_k\, dx + \frac12 \int_0^t \int_{\R^2} \frac{|\na n|^2}{1+n}\,\Phi_k\, dxds\\
&\le \int_{\R^2} \bkt{(1+n_0)\log(1+n_0) - n_0}\Phi_k\, dx + C\int_0^T\int_{\R^2} n(t)\Phi_k\, dx\\
&\quad + C2^{-2k}\bke{\norm{n_0}_1^6 + \norm{n_0}_1^3 + \norm{n_0}_1^2 + \norm{n_0}_1 + 1}T + C2^{-3k} \int_0^T \norm{u(t)}_3^3\, dt.
}
Note that we have $\norm{u}_{L^6(0,T;L^3(\R^2))}\le C(\norm{n_0}_1, \norm{n_0\log n_0}_1, \norm{u_0}_2)$ by \thref{u-regularity}. Then
\EQ{\label{est-u33}
\int_0^T \norm{u(t)}_3^3\, dt \le \int_0^T \bke{\frac{\norm{u(t)}_3^6}2 + \frac12} dt \le C(\norm{n_0}_1, \norm{n_0\log n_0}_1, \norm{u_0}_2)\, T.
}
Adding \eqref{NO-eq3.23} for $k\ge N$ and using $\sum_{k=N}^\infty \Phi_k(x)=1$ for $|x|\ge2^N$, we have
\EQN{
&\int_{|x|\ge 2^N} \bkt{(1+n(t))\log(1+n(t)) - n(t)} dx + \frac12 \int_0^t \int_{|x|\ge 2^N} \frac{|\na n|^2}{1+n}\, \Phi_k\, dxds\\
&\le \int_{|x|\ge 2^{N-1}} \bkt{(1+n_0)\log(1+n_0) - n_0} dx + C\bke{\norm{n_0}_1^6 + \norm{n_0}_1^3 + \norm{n_0}_1^2 + \norm{n_0}_1 + 1}T\\
&\quad + C(\norm{n_0}_{1}, \norm{n_0\log n_0}_1,\norm{u_0}_2)\, T,
}
proving \eqref{NO-eq3.3} with $R_0=2^N$. Next, we have (see \cite[(3.24)]{NO-2016})
\EQ{\label{NO-eq3.24}
\int_{\R^2} n^2(t) \Phi_k\, dx \le 2\norm{n_0}_1 \bke{\int_{\R^2} \frac{|\na n|^2}{1+n}\, \Phi_k\, dx} + C2^{-2k}\norm{n_0}_1^2 + 4\int_{\R^2} n(t)\Phi_k\, dx.
}
Using \eqref{NO-eq3.23} to the first term on the right hand side of \eqref{NO-eq3.24} and adding these for all $k\ge N$, we conclude that
\EQN{
\int_0^T \int_{|x|\ge 2^N} n^2\, dxdt &\le 2 \norm{n_0}_1 \int_0^T \int_{|x|\ge 2^N} \frac{|\na n|^2}{1+n}\, dxdt + C(\norm{n_0}_1^2+\norm{n_0}_1)T\\
& \le 4 \norm{n_0}_1 \int_{|x|\ge 2^{N-1}} \bkt{(1+n_0)\log(1+n_0) - n_0} dx\\
& \quad + C \norm{n_0}_1 \bke{\norm{n_0}_1^6 + \norm{n_0}_1^3 + \norm{n_0}_1^2 + \norm{n_0}_1 + 1}T\\
& \quad + C(\norm{n_0}_{1}, \norm{n_0\log n_0}_1, \norm{u_0}_2) \norm{n_0}_1\,T 
+ C(\norm{n_0}_1^2+\norm{n_0}_1)\,T.
}
This proves \eqref{NO-eq3.4}.
\end{proof}

\medskip

\subsection{A priori estimates for exterior regions}\label{NO-sec3.2}

Let $R_0>0$ be given in \thref{NO-prop3.2}, and choose it larger so that $n_0(x)\le 1$ for $|x|>R_0$. Then
\[
\sup_{0<t<T} \int_{|x|\ge R_0} (1+n(t)) \log(1+n(t))\, dx \le C(n_0,u_0,T),
\]
\[
\int_0^T \int_{|x|\ge R_0} \frac{|\na n|^2}{1+n}\, dxdt \le C(n_0,u_0,T),
\]
\[
\int_0^T \int_{|x|\ge R_0} n^2\, dxdt \le C(n_0,u_0,T),
\]
and for $R\ge R_0$ and $p\in[1,\infty)$
\[
\int_{|x|\ge R} n_0^p(x)\,dx \le \int_{|x|\ge R} n_0(x)\,dx.
\]
For $R>1$, let $\Phi_R\in C^\infty(\R^2)$ be such that $\Phi_R(x)\in[0,1]$, $\Phi_R(x)=1$ for $|x|\ge R$, $\textup{supp}\,\Phi_R\subset\R^2\setminus \overline{B_{R/2}(0)}$,
and
\[
|\na \Phi_R| \le CR^{-1}\Phi_R^{5/6},\qquad 
|\na(\Phi_R^{1/2})|\le CR^{-1}\Phi_R^{1/3},\qquad
|\na^2\Phi_R| \le CR^{-2}\Phi_R^{2/3},
\]
where $C>0$ is a constant independent of $R$. Then
\[
\textup{supp}\,(\na \Phi_R) \subset A_R^*:=\{x\in\R^2: R/2\le |x|\le R\}.
\]

\medskip

\begin{lem}\thlabel{NO-lem3.4}
For any $R\ge 2R_0$,
\EQ{\label{NO-eq3.28}
\sup_{0<t<T} \int_{|x|\ge R} n^2(t)\, dx + \frac12 \int_0^T \int_{|x|\ge R} |\na n|^2\, dxdt
\le C(\norm{n_0}_1,\norm{n_0\log n_0}_1,\norm{u_0}_2,T,R).
}
\end{lem}
\begin{proof}
For fixed $R\ge 2R_0$, we let $c_m(t) = c(t) - [c(t)]_B$, where $[c(t)]_B = (1/|B|) \int_B c(t)\,dx$, $B=B_{2R}(0)$. Then $\norm{c_m}_p\lec \norm{c}_{BMO}$ for $1<p<\infty$ by \thref{BMO-def-est}. Multiplying the equation of $n$ by $n\Phi_R$, integrating on $\R^2$, and performing integration by parts, we have
\EQN{
\frac12&\,\frac{d}{dt} \int_{\R^2} n^2(t)\Phi_R\, dx + \int_{\R^2} |\na n(t)|^2\Phi_R\, dx\\
&= -\int_{\R^2} nu\cdot\na n\Phi_R\, dx + \int_{\R^2} n(\De n)\Phi_R\, dx - \int_{\R^2} n\na\cdot(n\na c)\Phi_R\, dx + \int_{\R^2} |\na n|^2\Phi_R\, dx\\
&= -\int_{\R^2} nu\cdot\na n\Phi_R\, dx + \frac12 \int_{\R^2} n^3\Phi_R\, dx + \frac12 \int_{\R^2} n^2(1-c_m)\De\Phi_R\, dx - \int_{\R^2} n\na n\cdot\na \Phi_R c_m\, dx,
}
where
\[
\frac12 \int_{\R^2} n^2(1-c_m)\De\Phi_R\, dx 
\le \frac14 \int_{\R^2} n^3\Phi_R\, dx + CR^{-4}(\norm{c}_{\text{BMO}}^3+1)
\]
and
\[
- \int_{\R^2} n\na n\cdot\na \Phi_R c_m\, dx 
\le \frac13 \int_{\R^2} n^3\Phi_R\, dx + \frac12 \int_{\R^2} |\na n|^2\Phi_R\, dx + CR^{-4}\norm{c}_{\text{BMO}}^6.
\]
Hence
\EQN{
\frac{d}{dt} \int_{\R^2}& n^2\Phi_R\, dx + \int_{\R^2} |\na n|^2\Phi_R\, dx \\
&\le -\int_{\R^2} nu\cdot\na n\Phi_R\, dx + \frac{13}6 \int_{\R^2} n^3\Phi_R\, dx + CR^{-4}\bke{\norm{c}_{\text{BMO}}^6 + \norm{c}_{\text{BMO}}^3 + 1},
}
where 
\EQN{
- \int_{\R^2} nu\cdot\na n \Phi_R\, dx 
= -\frac12 \int_{\R^2} u\cdot\na (n^2) \Phi_R\, dx
=&~ \frac12 \int_{\R^2} n^2 u\cdot\na \Phi_R\, dx\\
\le&~ \frac12\, CR^{-1} \int_{A_R^*} n^2 |u|\Phi_R\, dx\\
\le&~ CR^{-3} \norm{u}_3^3 + \frac56 \int_{\R^2} n^3\Phi_R\, dx
}
since $u$ is divergence-free.
Thus,
\EQN{
\frac{d}{dt} \int_{\R^2}& n^2\Phi_R\, dx + \int_{\R^2} |\na n|^2\Phi_R\, dx \\
&\le 3 \int_{\R^2} n^3\Phi_R\, dx + CR^{-4}\bke{\norm{c}_{\text{BMO}}^6 + \norm{c}_{\text{BMO}}^3 + 1} + CR^{-3} \norm{u}_3^3,
}
and from a similar argument as that of \cite[(3.31)]{NO-2016} we have
\EQN{
\frac{d}{dt} \int_{\R^2} &n^2\Phi_R\, dx + \bke{1-\ve \int_{A_R^*} (1+n)\log(1+n)\, dx} \int_{\R^2} |\na n|^2 \Phi_R\, dx\\
& \le C\bke{\int_{\R^2} n^{3/2}|\na \Phi_R^{1/2}|\, dx}^2 + C(\ve) \int_{\R^2} n\Phi_R\, dx + CR^{-4}\bke{\norm{c}_{\text{BMO}}^6 + \norm{c}_{\text{BMO}}^3 + 1}\\
&\quad + CR^{-3}\norm{u}_3^3.
}
Choosing $\ve>0$ such that
\[
1-\ve \int_{A_R^*} (1+n)\log(1+n)\, dx = \frac12,
\]
we have
\EQN{
\frac{d}{dt} \int_{\R^2} &n^2\Phi_R\, dx + \frac12 \int_{\R^2} |\na n|^2 \Phi_R\, dx\\
& \le CR^{-2}\norm{n_0}_1 \int_{|x|\ge R/2} n^2\, dx + C(\ve)\norm{n_0}_1 + C(\norm{n_0}_1) R^{-4} + CR^{-3}\norm{u}_3^3.
}
Recall from \eqref{est-u33} that $\int_0^T \norm{u(t)}_3^3\, dt \le C(\norm{n_0}_1, \norm{n_0\log n_0}_1, \norm{u_0}_2)\, T$. Integrating the above inequality from $0$ to $t$ with respect to the time variable, we have
\EQN{
\int_{\R^2} n^2(t)\Phi_R\, dx + \frac12 \int_0^t \int_{\R^2} |\na n|^2\Phi_R\, dxds
& \le \int_{\R^2} n_0^2\Phi_R\, dx + C(\norm{n_0}_1,\norm{n_0\log n_0}_1,\norm{u_0}_2,T,R)\\
& \le \int_{\R^2} n_0\Phi_R\, dx + C(\norm{n_0}_1,\norm{n_0\log n_0}_1,\norm{u_0}_2,T,R).
}
This proves \eqref{NO-eq3.28}.
\end{proof}

\medskip

\begin{lem}\thlabel{NO-lem3.5}
For any $R\ge 2^2 R_0$,
\[
\int_0^T \int_{|x|\ge R} n^4\, dxdt \le C(\norm{n_0}_1, \norm{n_0\log n_0}_1, \norm{u_0}_2, T,R).
\]
\end{lem}
\begin{proof}
The lemma is a consequence of \thref{NO-lem3.4} following the same proof of \cite[Lemma 3.5]{NO-2016}.
\end{proof}

\medskip

\begin{lem}\thlabel{NO-lem3.6}
For any $R\ge 2^3 R_0$,
\[
\sup_{0<t<T} \int_{|x|\ge R} n^3(t)\, dx + \int_0^T \int_{|x|\ge R} |\na n^{3/2}|^2 \Phi_R\, dxdt
\le C(\norm{n_0}_1, \norm{n_0\log n_0}_1, \norm{u_0}_2, T,R).
\]
\end{lem}
\begin{proof}
Multiplying the equation of $n$ by $n^2\Phi_R$, integrating over $\R^2$, and using integration by parts and the fact that $u$ is divergence-free, we have
\EQN{
\frac13\, \frac{d}{dt}& \int_{\R^2} n^3 \Phi_R\, dx + \frac89 \int_{\R^2} |\na n^{3/2}|^2\Phi_R\, dx\\
&= -\int_{\R^2} (u\cdot\na n) n^2\Phi_R\, dx + \frac13 \int_{\R^2} n^3 \De\Phi_R\, dx - \int_{\R^2} n^2\na \cdot(n\na c_m)\Phi_R\, dx\\
&= \frac13 \int_{\R^2} n^3 u\cdot\na\Phi_R\, dx + \frac13 \int_{\R^2} n^3 \De\Phi_R\, dx + \frac13 \int_{\R^2} n^3\na c_m\cdot\na\Phi_R\, dx + \frac23 \int_{\R^2} n^4\Phi_R\, dx,
}
where $c_m(t) = c(t) - [c(t)]_B$, $[c(t)]_B = (1/|B|) \int_B c(t)\,dx$, $B=B_{2R}(0)$ satisfies $\norm{c_m}_p\lec \norm{c}_{BMO}$ by \thref{BMO-def-est}. So
\EQN{
\frac{d}{dt}& \int_{\R^2} n^3 \Phi_R\, dx + \frac83 \int_{\R^2} |\na n^{3/2}|^2\Phi_R\, dx\\
&= \int_{A_R^*} n^3 u\cdot\na\Phi_R\, dx + \int_{\R^2} n^3 \De\Phi_R\, dx + \int_{\R^2} n^3\na c_m\cdot\na\Phi_R\, dx + 2\int_{\R^2} n^4\Phi_R\, dx\\
&\le CR^{-1} \int_{A_R^*} n^3|u|\Phi_R^{5/6}\, dx + CR^{-2} \int_{|x|\ge R/2} n^3\, dx + \int_{\R^2} n^3 \na c_m\cdot\na\Phi_R + 2\int_{\R^2} n^4\, dx
\\
&\le CR^{-1} \norm{u(t)}_4^4 + CR^{-1} \int_{|x|\ge R/2} n^4\, dx + F(t),
}
where
\EQN{
F(t) =&~ CR^{-2} \int_{|x|\ge R/2} n^3\, dx + C\int_{|x|\ge R/2} n^4\, dx\\
&~ + CR^{-6}\bke{\norm{c(t)}_{\text{BMO}}^8 + \norm{c(t)}_{\text{BMO}}^4},
}
and $\norm{u}_{L^4(\R^2\times(0,T))}\le C(\norm{n_0}_1, \norm{n_0\log n_0}_1, \norm{u_0}_2)$ by \thref{u-regularity}. By \thref{NO-lem3.4} and \thref{NO-lem3.5} we have
\[
\int_0^T F(t)\, dx \le C(\norm{n_0}_1, \norm{n_0\log n_0}_1,\norm{u_0}_2, T,R).
\]
Therefore,
\EQN{
\int_{\R^2} n^3(t)\Phi_R\, dx +& \int_0^t \int_{\R^2} |\na n^{3/2}|^2 \Phi_R\, dxds\\
&\qquad \le \int_{\R^2} n_0^3\Phi_R\, dx + C(\norm{n_0}_1, \norm{n_0\log n_0}_1, \norm{u_0}_2, T,R)\\
&\qquad \le \int_{\R^2} n_0\Phi_R\, dx + C(\norm{n_0}_1, \norm{n_0\log n_0}_1, \norm{u_0}_2,T,R).
}
This completes the proof of the lemma.
\end{proof}

\medskip

\subsection{A priori estimates for annuli}
Let $R_0>0$ be given in \thref{NO-prop3.2}. Throughout the rest of this paper, we denote
\EQ{\label{def-AR}
A_R = \{x\in\R^2:R/2\le|x|\le2R\}.
}
For $R\ge1$ we let $\tilde \Phi_R\in C^\infty_0(\R^2)$ be such that $\tilde \Phi_R(x)\in[0,1]$, $\tilde \Phi_R(x)=1$ for $R/2\le |x|\le 2R$, 
\[
\text{supp}\, \tilde \Phi_R \subset \tilde A_R := \{x\in\R^2:R/3\le|x|\le3R\},
\]
and
\[
|\na \tilde \Phi_R| \le CR^{-1} \tilde \Phi_R^{5/6},\qquad 
|\na \tilde \Phi_R^{1/2}| \le CR^{-1} \tilde \Phi_R^{1/3},\qquad
|\na^2 \tilde \Phi_R|\le CR^{-2} \tilde \Phi_R^{2/3},
\]
where $C>0$ is independent of $R$. We put $m(x,t) = n(x,t) \tilde \Phi_R(x)$, which satisfies $m\in C([0,T);L^2(\R^2)) \cap L^\infty(0,T;L^2(\R^2))$ and 
\EQ{\label{NO-eq3.38}
\pd_tm - \De m = f,\qquad 0<t<T,\quad x\in\R^2,
}
where
\EQ{\label{NO-eq3.39}
f = -u\cdot\na n \tilde \Phi_R - (2\na \tilde \Phi_R + \tilde \Phi_R\na c)\cdot\na n - (\De \tilde \Phi_R)n + \tilde \Phi_R n^2.
}
Denote $w(x,t) = u(x,t) \tilde \Phi_R(x)$, then 
\EQ{\label{NO-eq3.38-u}
\pd_t w - \De w = g,\qquad 0<t<T,\quad x\in\R^2,
}
where
\EQ{\label{NO-eq3.39-u}
g = -(u\cdot\na)u \tilde \Phi_R - \na P \tilde \Phi_R + n \na c \tilde \Phi_R - 2 \na u\cdot \na \tilde \Phi_R - u(\De\tilde \Phi_R).
}

\medskip

\begin{lem}\thlabel{NO-lem3.7}
There exists $R_1\ge 2^5 R_0$ such that for any $R\ge R_1$, $\na c$ and $c_m$ are bounded on $A_R\times(0,T)$, where $c_m(t) = c(t) - \bkt{c(t)}_B$ in which $\bkt{c(t)}_B = (1/|B|)\int_B c(t)\, dx$, $B=B_{2R}(0)$.
\end{lem}
\begin{proof}
The lemma is a consequence of \thref{NO-lem3.6} following the same proof of \cite[Lemma 3.7]{NO-2016}.
\end{proof} 

\medskip

\begin{lem}\thlabel{lem-u-bound-annulus}
There exists $R_2\ge 3 R_1$ such that for any $R\ge R_2$, $u$ is bounded on $A_R\times(0,T)$.
\end{lem}
\begin{proof}
By \thref{u-regularity-1} and \thref{NO-lem3.4}, $g$ defined by \eqref{NO-eq3.39-u} is in $L^2(0,T;L^2(\R^2))$. The rest of the proof is identical to that of \cite[Lemma 3.8]{NO-2016}.
\end{proof}

\medskip

\begin{lem}\thlabel{NO-lem3.8}
There exists $R_3\ge 3 R_2$ such that for any $R\ge R_3$, $n$ is bounded on $A_R\times(0,T)$.
\end{lem}
\begin{proof}
By Lemmas in Subsection \ref{NO-sec3.2}, \thref{NO-lem3.7} and \thref{lem-u-bound-annulus}, $f$ defined by \eqref{NO-eq3.39} is in $L^2(0,T;L^2(\R^2))$. The rest of the proof is identical to that of \cite[Lemma 3.8]{NO-2016}.
\end{proof}

\medskip

\begin{lem}\thlabel{NO-lem3.9}
There exists $R_4\ge 3^3 R_3$ such that for any $R\ge R_4$, $\na n$ is bounded on $A_R\times(0,T)$.
\end{lem}
\begin{proof}~\\
{\bf Step 1.} We claim that
\EQ{\label{NO-eq3.43}
\sup_{0<t<T} \norm{\na n(t)}_{L^2(A_R)} < \infty.
}
Indeed, denoting $c_m(t) = c(t) - [c(t)]_B$, $B=B_{2R}(0)$,
we have
\EQN{
&\frac12\, \frac{d}{dt} \int_{\R^2} |\na n|^2 \tilde \Phi_R\, dx\\
&= \int_{\R^2} u\cdot\na n(\De n) \tilde \Phi_R\, dx - \int_{\R^2} |\De n|^2 \tilde \Phi_R\, dx - \int_{\R^2} \na n\cdot\na c_m(\De n) \tilde \Phi_R\, dx - \int_{\R^2} n^2 \De n \tilde \Phi_R\, dx\\
&\quad + \int_{\R^2} (u\cdot\na n)(\na n\cdot\na \tilde \Phi_R)\, dx - \int_{\R^2} \De n\na n\cdot\na \tilde \Phi_R\, dx + \int_{\R^2} (\na n\cdot\na c_m)(\na n\cdot\na \tilde \Phi_R)\, dx\\
&\quad - \int_{\R^2} n^2\na n\cdot \tilde \Phi_R\, dx\\
&\le \int_{\R^2} u\cdot\na n(\De n) \tilde \Phi_R\, dx - \frac34 \int_{\R^2} |\De n|^2 \tilde \Phi_R\, dx + C\norm{\na c_m}_{L^\infty(\tilde{A}_R)}^2 \int_{\tilde{A}_R} |\na n|^2\, dx + C \int_{\tilde{A}_R} n^4\, dx\\
&\quad + \int_{\R^2} (u\cdot\na n)(\na n\cdot\na \tilde \Phi_R)\, dx + \frac14 \int_{\R^2} |\De n|^2 \tilde \Phi_R\, dx + C(\norm{\na c_m}_{L^\infty(\tilde{A}_R)} + 1) \int_{\tilde{A}_R} |\na n|^2\, dx\\
&\quad + C\int_{\tilde{A}_R} n^4\, dx.
}
Since
\EQN{
\int_{\R^2} u\cdot\na& n(\De n) \tilde \Phi_R\, dx\\ 
&\le \int_{\R^2} |u||\na n||\De n| \tilde \Phi_R\, dx\\
&\le C(\ve) \norm{u(t)}_{L^\infty(\tilde{A}_R)} \int_{\tilde{A}_R} |\na n|^2\, dx + \ve \norm{u(t)}_{L^\infty(\tilde{A}_R)} \int_{\R^2} |\De n|^2 \tilde \Phi_R\, dx\\
&\le C(\ve)\, \norm{u}_{L^\infty(\tilde{A}_R\times (0,T))} \int_{\tilde{A}_R} |\na n|^2\, dx  + \ve\, \norm{u}_{L^\infty(\tilde{A}_R\times (0,T))} \int_{\R^2} |\De n|^2 \tilde \Phi_R\, dx\\
&\le C(\ve)\, \norm{u}_{L^\infty(\tilde{A}_R\times (0,T))} \int_{\tilde{A}_R} |\na n|^2\, dx + \frac14 \int_{\R^2} |\De n|^2 \tilde \Phi_R\, dx
}
by choosing $\ve>0$ sufficiently small, and
\EQN{
\int_{\R^2} (u\cdot\na n) (\na n\cdot\na \tilde \Phi_R)\, dx
&\le CR^{-1} \int_{\R^2} |u||\na n|^2 \tilde \Phi_R^{5/6}\\
&\le CR^{-1} \norm{u(t)}_{L^\infty(\tilde{A}_R)} \int_{\tilde{A}_R} |\na n|^2\\
&\le CR^{-1} \norm{u}_{L^\infty(\tilde{A}_R\times (0,T))} \int_{\tilde{A}_R} |\na n|^2,
}
we obtain
\EQN{
\frac{d}{dt}& \int_{\R^2} |\na n|^2 \tilde \Phi_R\, dx + \int_{\R^2} |\De n|^2 \tilde \Phi_R\, dx\\
& \le C\bkt{\norm{\na c_m}_{L^\infty(\tilde{A}_R)}^2 + \norm{\na c_m}_{L^\infty(\tilde{A}_R)} + (C(\ve) + CR^{-1})\norm{u}_{L^\infty(\tilde{A}_R\times (0,T))} + 1}\int_{\tilde{A}_R} |\na n|^2\, dx\\
&\quad +C\int_{\tilde{A}_R} n^4\, dx.
}
The claim of {\bf Step 1}, \eqref{NO-eq3.43} is a direct consequence of the boundedness of $\nabla c$, $n$ and $u$ in the annulus (Lemma \ref{NO-lem3.7}, Lemma \ref{lem-u-bound-annulus}, Lemma \ref{NO-lem3.8}) and a use of Gr\"onwall's inequality.

\medskip

\noindent{\bf Step 2.} 
Let $f$ be as in \eqref{NO-eq3.39}. By \eqref{NO-eq3.43} and 
\EQN{
\norm{(u\cdot\na n) \tilde \Phi_R}_2 
&\le \norm{u(t)}_{L^\infty(\tilde{A}_R)} \bke{\int_{\R^2} |\na n|^2 \tilde \Phi_R\, dx}^{1/2}\\
&\le \norm{u}_{L^\infty(\tilde{A}_R\times (0,T))} \bke{\int_{\R^2} |\na n|^2 \tilde \Phi_R\, dx}^{1/2},
}
we have
\[
\sup_{0<t<T} \norm{f(t)}_2<\infty
\]
in view of \thref{lem-u-bound-annulus}. The rest of the proof is identical to that of \cite[Lemma 3.9]{NO-2016}.
\end{proof}

\medskip

\begin{lem}\thlabel{NO-lem3.10}
There exists $R_5\ge 3 R_4$ such that for any $R\ge R_5$, $\na^2 c$ is bounded on $A_R\times(0,T)$.
\end{lem}
\begin{proof}
The lemma is a consequence of \thref{NO-lem3.8} and \thref{NO-lem3.9} following the same proof of \cite[Lemma 3.10]{NO-2016}.
\end{proof}

\medskip

\begin{lem}\thlabel{NO-lem3.11}
There exists $R_6\ge 3^2 R_5$ such that for any $R\ge R_6$, $\na^2 n$ is bounded on $A_R\times(0,T)$.
\end{lem}
\begin{proof}
Let $f$ be as in \eqref{NO-eq3.39}. Since $\sup_{0<t<T}\norm{u(t)\cdot\na n(t) \tilde \Phi_R}_\infty<\infty$, we have
\[
\sup_{0<t<T}\norm{f(t)}_\infty<\infty.
\]
The rest of the proof is identical to that of \cite[Lemma 3.11]{NO-2016}.
\end{proof}

\medskip

\begin{prop}\thlabel{NO-prop3.12}
There exists $\tilde R_0>R_0$ such that for any $R\ge \tilde R_0$, the followings hold:
\EN{
\item [(i)] $\pd_t^k \na_x^l n$ $(0\le 2k+l\le 2)$ are bounded on $A_R\times(0,T)$.
\item [(ii)] $c_m$, $\na_x^l c$ $(1\le l\le2)$, and $u$ are bounded on $A_R\times(0,T)$.
\item [(iii)] there exist $x_0$ satisfying $R<|x_0|<2R$, and $\ve_0>0$, $\de>0$ such that
\[
n(x,t) \ge \de \qquad \text{ for } |x-x_0|\le \ve_0,\,0\le t<T.
\]
}
\end{prop}
\begin{proof}
The assertions (i) and (ii) follow from the preceding lemmas. 

\medskip

To prove (iii), we claim that $\na^l n$ $(0\le l\le 2)$ are uniformly H\"{o}lder continuous on $\tilde{A}_R^* \times [T/2,T)$, where $\tilde{A}_R^*:=\{x\in\R^2: 2R/3\le |x|\le 8R/3\}$.
\EN{
\item [(a)] Since $n$ satisfies $\pd_t n - \De n = -\na\cdot(n(u+\na c))$ and $n(u+\na c)$ is bounded on $\{x\in\R^2:R/3<|x|<3R\}\times(0,T)$, applying local Schauder estimates for parabolic equations, we see that $n$ is uniformly H\"{o}lder continuous on $ \tilde{A}_R^*\times [T/2,T)$.
\item [(b)] Similarly, by $\pd_t (\na n) - \De (\na n) = -\na\cdot(\na(n(u+\na c)))$ and the boundedness of $\na(n(u+\na c))$ in $\{R/3<|x|<3R\}\times (0,T)$, the uniformly H\"{o}lder continuity of $\na n$ in $\tilde{A}_R^*\times [T/2,T)$ is deduced.
\item [(c)] For $\na^2 n$, observe that
\[
\pd_t(\pd_k\pd_j n) - \De (\pd_k\pd_j n) + \na c\cdot\na (\pd_k\pd_j n) - 2n (\pd_k\pd_j n) = \pd_k g_1 + g_2
\]
where $g_1 = - \pd_j(u\cdot\na n) - \na n\cdot\na(\pd_j c)$ and $g_2 = -\na \pd_j n \cdot \na \pd_k c + 2 \pd_kn\pd_jn$. Hence, by the boundedness of $g_1$ and $g_2$ on $\{R/3<|x|<3R\}\times(0,T)$, $\pd_k\pd_j n$ is uniformly H\"{o}lder continuous on $\tilde{A}_R^*\times [T/2,T)$.
}

\medskip

The rest of the proof is identical to that of \cite[Proposition 3.12]{NO-2016} given that $u+\na c$ is bounded on $\{x\in\R^2: R<|x|<2R\}\times [T/2,T)$.
\end{proof}

\medskip

\subsection{Entropy estimates for interior regions}

\medskip

We now derive the a priori estimate for small $|x|$. Define $\Psi_R(x) = \Psi(R^{-1}x)$ where $\Psi\in C^\infty(\R^2)$, $\Psi(x)\in[0,1]$, $\Psi(x)=1$ for $|x|\le1$, $\textup{supp}\,\Psi\subset B_2(0)$.
It follows from the definition of $\Psi_R$ that
\[
\text{supp}\,[\na\Psi_R] \subset A_R,\qquad \text{supp}\,[\De\Psi_R] \subset A_R,
\]
where $A_R$ is an annulus defined in \eqref{def-AR}.

\medskip

We define
\EQN{
H_{\text{int}}(t;R) :=& \int_{\R^2} \bkt{(1+n(t))\log(1+n(t)) - n(t)} \Psi_R\, dx - \frac12 \int_{\R^2} n(t) c_m(t) \Psi_R\, dx\\
& + \frac12 \int_{\R^2} |u|^2 \Psi_R\, dx,
}
where $c_m(t) = c(t) - [c(t)]_B$, $[c(t)]_B=(1/|B|)\int_B c(t)\, dx$, $B=B_{2R}(0)$.

\medskip

\begin{lem}\thlabel{NO-lem3.13}
\EQN{
\frac{d}{dt}\, H_{\text{int}}(t;R) + \int_{\R^2}& n|\na \log(1+n(t)) - \na c_m(t)|^2 \Psi_R\, dx + \int_{\R^2} |\na \log(1+n(t))|^2 \Psi_R\,dx \\
&+ \int_{\R^2} |\na u|^2 \Psi_R\, dx = \int_{\R^2} n(t) \log(1+n(t)) \Psi_R\, dx + F_1(t) + F_2(t) + F_3(t),
}
where
\EQN{
F_1(t) =& \int_{\R^2} \bkt{(1+n(t))\log(1+n(t)) - n(t)} \De\Psi_R\, dx\\
& - \int_{\R^2} n(t) c_m(t) \na c_m(t) \cdot\na \Psi_R\, dx + \int_{\R^2} c_m(t)\na n(t)\cdot \na\Psi_R\, dx\\
&~ + \int_{\R^2}\bkt{n(t)\log(1+n(t)) - \log(1+n(t))} \na c_m(t)\cdot\na\Psi_R\, dx,
}
\[
F_2(t) = -\int_{\R^2} \pd_tc_m(t) \na c_m(t)\cdot \na\Psi_R\, dx - \frac14\,\frac{d}{dt} \int_{\R^2} |c_m(t)|^2 \De \Psi_R\, dx,
\]
and
\EQ{\label{eq-F3}
F_3(t) 
=& \int_{\R^2} \bkt{(1+n)\log(1+n) - n} u\cdot\na \Psi_R\, dx + \frac12\int_{\R^2} |u|^2 u\cdot\na\Psi_R\, dx - \int_{\R^2} u\cdot\na P \Psi_R\, dx\\
& - \int_{\R^2} (u\cdot\na)u\cdot\na \Psi_R\, dx - \int_{\R^2} nc_mu\cdot\na\Psi_R\, dx.
}
\end{lem}
\begin{proof}
The proof is identical to that of \cite[Lemma 3.13]{NO-2016} with an addition term $F_3(t)$ generated by the coupling of the velocity field $u$.
\end{proof}

\medskip

\begin{lem}\thlabel{NO-lem3.14}
For $F_1(t)$ in \thref{NO-lem3.13}, it holds that
\[
\sup_{0<t<T}|F_1(t)|<\infty.
\]
\end{lem}
\begin{proof}
The proof is identical to that of \cite[Lemma 3.14]{NO-2016}.
\end{proof}

\medskip

\begin{lem}\thlabel{NO-lem3.15}
For $F_2(t)$ in \thref{NO-lem3.13}, it holds that
\[
\sup_{0<t<T}\left|\int_0^t F_2(s)\, ds\right|<\infty.
\]
\end{lem}
\begin{proof}
We define $\tilde\Psi_R$ as in the proof of \cite[Lemma 3.15]{NO-2016}. Indeed, $\tilde\Psi_R(x):=\Psi_R(x) J_R(x)$, where $J_R\in C_0^\infty(\R^2)$, $J_R(x)\in[0,1]$, $J_R(x)=1$ for $R\le|x|\le2R$, 
\[
\textup{supp}\,J_R\subset B_{3R}(0)\setminus \overline{B_{R/2}(0)},
\] 
and
\[
|\na J_R(x)|\le CR^{-1},\qquad |\De J_R(x)|\le CR^{-2}.
\]
The support of $\tilde\Psi_R$ is contained in $A_R$, where $A_R$ is an annulus defined in \eqref{def-AR}. In terms of $\tilde\Psi_R$, $F_2$ can be rewritten as 
\EQN{
F_2(t)
&= \frac14\, \frac{d}{dt} \bket{\int_{\R^2}c_m^2\De\tilde\Psi_R\, dx - 2|\na c_m|^2\tilde\Psi_R\, dx}\\
&\quad + \int_{\R^2} \bkt{\De n - \na n\cdot(\na c_m + u) + n^2} c_m\tilde\Psi_R\, dx\\
&=: F_{21}(t) + F_{22}(t),
}
where $F_{21}(t)$ is bounded as shown in the proof of \cite[Lemma 3.15]{NO-2016}. For $F_{22}(t)$, we rewrite it as 
\[
F_{22}(t) 
= -2\int_{\R^2} \na n\cdot(\na c_m + u)\tilde\Psi_R\, dx - \int_{\R^2} \na n\cdot\na \tilde\Psi_R c_m\, dx + \int_{\R^2} n^2c_m\tilde\Psi_R\, dx.
\]
Since the integrands in $F_{22}(t)$ are bounded on $(0,T)\times\tilde{A}_R$ by (i) and (ii) of \thref{NO-prop3.12}, we also have
\[
\sup_{0<t<T} |F_{22}(t)|<\infty.
\]
This completes the proof of \thref{NO-lem3.15}.
\end{proof}

\medskip

\begin{lem}\thlabel{lem-F3}
For $F_3(t)$ in \thref{NO-lem3.13}, it holds that
\[
\sup_{0<t<T}\left|\int_0^t F_3(s)\, ds\right|<\infty.
\]
\end{lem}

\begin{proof}
Note that
\[
- \int_{\R^2} u\cdot\na P \Psi_R\, dx \le \frac12\,\norm{u(t)}_2^2 + \frac12\, \norm{\na P(t)}_2^2 \in L^1(0,T)
\]
by \thref{u-regularity-1}, and all the other terms in \eqref{eq-F3} containing $\na\Psi_R$ are bounded by (i) and (ii) of \thref{NO-prop3.12}. Therefore \thref{lem-F3} is established.
\end{proof}


\medskip

The following proposition of interior estimate and the exterior estimate in \thref{NO-prop3.2} give the boundedness of the full modified entropy, which completes the proof of \thref{NO-thm3.1}.

\medskip

\begin{prop}\thlabel{NO-prop3.16}
Assume $\int_{\R^2} n_0\, dx \le 8\pi$. Then
\[
\sup_{0<t<T} \int_{|x|\le 4\tilde R_0} (1+n(t)) \log(1+n(t))\, dx < \infty,
\]
where $\tilde{R}_0$ is the one determined in \thref{NO-prop3.12}.
\end{prop}
\begin{proof}
Denote $R=4\tilde{R}_0$. It follows from \thref{NO-lem3.13} that
\EQN{
\frac{d}{dt}\, H_{\text{int}}(t;R)\le \int_{\R^2} (1+n(t))\log(1+n(t)) \Psi_R\, dx + F_1(t) + F_2(t) + F_3(t).
}
For $0<a<1$,
\[
\frac12\, nc_m 
\le (1-a)n\,\frac{|c_m|}{2(1-a)}
\le (1-a)(1+n)\log(1+n) + (1-a) \exp\bket{\frac{|c_m|}{2(1-a)}},
\]
where the Fenchel--Young inequality is used in the last inequality.
Then it follows that
\[
(1+n)\log(1+n) \le \frac1a \bke{(1+n)\log(1+n) - \frac12\, nc_m} + \frac{1-a}a\, \exp\bket{\frac{|c_m|}{2(1-a)}}.
\]
Then
\EQ{\label{NO-eq3.54}
\int_{\R^2} &(1+n(t)) \log(1+n(t)) \Psi_R\, dx\\
&\le \frac1a\, H_{\text{int}}(t;R) + \frac1a \int_{\R^2} n(t) \Psi_R\, dx + \frac{1-a}a \int_{\R^2} \exp\bket{\frac{|c_m|}{2(1-a)}} \Psi_R\, dx\\
&\le \frac1a\, H_{\text{int}}(t;R) + \frac1a\, \norm{n_0}_1 + F_4(t),
}
where
\[
F_4(t) = \frac{1-a}a \int_{\R^2} \exp\bket{\frac{|c_m|}{2(1-a)}} \Psi_R\, dx.
\]
 Therefore,
\EQ{\label{NO-eq3.55}
\frac{d}{dt}\, H_{\text{int}}(t;R) \le \frac1a\, H_{\text{int}}(t;R) + F(t),
}
where 
\[
F(t) = F_1(t) + F_2(t) + F_3(t) + F_4(t) + \frac1a\, \norm{n_0}_1.
\]

We claim 
\EQ{\label{NO-eq3.56}
\sup_{0<t<T} F_4(t)<\infty.
}
To prove this claim, by (iii) of \thref{NO-prop3.12} there exist $x_0\in\R^2$, $\ve_0>0$, $\de>0$ with $R+\ve_0<|x_0|<2R$ such that
\EQ{\label{NO-eq3.57}
n(x,t) \ge \de \qquad \text{ for } 0\le t<T,\ |x-x_0|\le \ve_0.
}
By (ii) of \thref{NO-prop3.12},
\[
C := \sup_{0<t<T} \norm{c_m(T)}_{L^\infty(R\le|x|\le2R)} < \infty.
\]
The function $c_m(t)$ satisfies
\[
-\De\bke{\frac{c_m(t)}{2(1-a)}} = \frac{n(t)}{2(1-a)},\qquad |x|<R
\]
and $\sup_{|x|=R}|c_m(t)|\le C$ $(0<t<T)$. By \eqref{NO-eq3.57}, we have
\EQN{
\int_{|x|\le R} n(t)\, dx 
&\le \int_{\R^2} n(t)\, dx - \int_{|x-x_0|\le \ve_0} n(t)\, dx\\
&\le 8\pi - \de\pi\ve_0^2 = \pi(8-\de\ve_0^2)
}
and hence
\[
\int_{|x|\le R} \frac{n(t)}{2(1-a)}\, dx \le \frac{\pi}{2(1-a)}\, (8-\de\ve_0^2).
\]
We choose $0<a<1$ such that 
\[
\frac{\pi}{2(1-a)}\, (8-\de\ve_0^2)<4\pi
\]
and apply \thref{NO-lem2.7} to get that for $0<t<T$,
\EQN{
\int_{|x|<R} \exp\bket{\frac{|c_m(t)|}{2(1-a)}} dx
& \le \frac{32\pi(1-a)R^2}{\de\ve_0^2 - 8a}\, \exp\bket{\sup_{|x|=R} \frac{|c_m(t)|}{2(1-a)}}\\
& \le \frac{32\pi(1-a)R^2}{\de\ve_0^2 - 8a}\, \exp\bket{\frac{C}{2(1-a)}}.
}
Hence
\[
\int_{\R^2} \exp\bket{\frac{|c_m(t)|}{2(1-a)}} \Psi_R\, dx 
\le \bke{\frac{32\pi(1-a)R^2}{\de\ve_0^2 - 8a} + 3\pi R^2} \exp\bket{\frac{C}{2(1-a)}},
\]
which implies \eqref{NO-eq3.56}.

By \eqref{NO-eq3.55}, we have
\[
H_{\text{int}}(t;R) 
\le H_{\text{int}}(0;R) + \frac1{a} \int_0^t H_{\text{int}}(s;R)\, ds + \sup_{0<t<T}\abs{\int_0^t F(s)\, ds}, \qquad 0<t<T
\]
and by \thref{NO-lem3.14}, \thref{NO-lem3.15}, \thref{lem-F3}, and \eqref{NO-eq3.56},
\[
\sup_{0<t<T}\abs{\int_0^t F(s)\, ds} < \infty.
\]
Applying the Gronwall inequality, we deduce that
\[
\sup_{0<t<T} H_{\text{int}}(t;R) < \infty.
\]
Therefore, by this estimate, \eqref{NO-eq3.54} and \eqref{NO-eq3.56}, we obtain
\[
\sup_{0<t<T} \int_{\R^2} (1+n(t))\log(1+n(t)) \Psi_R\, dx <\infty,
\]
which completes the proof of \thref{NO-prop3.16}.
\end{proof}

\bigskip

\section{Global existence of the Patlak--Keller--Segel--Navier--Stokes system (\ref{PKS-NS})}\label{S-global-exist}

\medskip

In this section, we prove our main result on global-in-time existence, \thref{NO-thm1.2}. 

\medskip

\begin{prop}\thlabel{prop-Linfty-bound}
Let $(n_0,u_0)$ satisfy \eqref{initial-n} and \eqref{initial-u} with $M:= \norm{n_0}_1 = 8\pi$. For a local-in-time mild solution $(n,u)$ on $[0,T)$ with initial data $(n_0,u_0)$ given in \thref{local-exist}. If $T<\infty$, it holds that for any $t_0\in(0,T)$
\[
\sup_{t\in(t_0,T)} \bke{ \norm{n(t)}_\infty + \norm{u(t)}_\infty } <\infty.
\]
\end{prop}
\begin{proof}
The proof is divided into five steps.

\noindent {\bf Step 1.} 
By testing \eqref{PKS-NS}$_1$ with $n$ and using the inequality \eqref{NO-2011-lem2.1-2}, we get for $t\in(t_0, T)$, where $t_0\in(0,T)$ is fixed, that
\EQN{
&\frac12\, \norm{n(t)}_2^2 + \int_{t_0}^t \norm{\na n(\tau)}_2^2\, d\tau\\
&~ = \frac12\, \norm{n(t_0)}_2^2 + \frac12 \int_{t_0}^t \norm{n(\tau)}_3^3\, d\tau\\
&~ \le \frac12\, \norm{n(t_0)}_2^2 + \frac{\ve}2 \bke{\sup_{t\in(t_0,T)} \norm{(1+n(t))\log(1+n(t))}_1} \int_{t_0}^t \norm{\na n(\tau)}_2^2\, d\tau + \frac{C(\ve)}2\,Mt
} 
since $\nabla\cdot u=0$. It follows from \thref{NO-thm3.1} that $\sup_{t\in(t_0,T)} \norm{(1+n(t))\log(1+n(t))}_1<\infty$. Thus we can choose $\ve>0$ sufficiently small so that for any $t_0\in(0,T)$
\EQ{\label{prop5.1-step1}
\sup_{t\in(t_0,T)}\norm{n(t)}_2<\infty
}
if $T<\infty$.

\medskip

\noindent {\bf Step 2.} Multiplying $\eqref{PKS-NS}_1$ with $n^{p-1}$, we have
\EQN{
\frac{d}{dt}\, \bke{\frac{n^p}p} + u\cdot\na\bke{\frac{n^p}p} 
= \na\cdot (n^{p-1}\na n) - \frac{4(p-1)}{p^2}\, |\na(n^{p/2})|^2 - \na\cdot(n\na c)n^{p-1}.
}
Integrating over $\R^2$, we obtain
\[
\frac1p\,\frac{d}{dt}\,\norm{n(t)}_p^p = -\frac{4(p-1)}{p^2} \int_{\R^2} |\na(n^{p/2})|^2\, dx + \frac{p-1}p \int_{\R^2} n^{p+1} \, dx.
\]
Taking $p=4$, 
\[
\frac14\,\frac{d}{dt}\,\norm{n(t)}_4^4 = -\frac34 \int_{\R^2} |\na(n^2)|^2 \, dx + \frac34 \int_{\R^2} n^5 \, dx.
\]
By Gagliardo-Nirenberg-Sobolev inequality and Young's inequality, we have
\EQN{
\norm{n}_5^5 = \norm{n^2}_{5/2}^{5/2} 
\le C \norm{n^2}_1 \norm{\na(n^2)}_2^{3/2}
=&~ C \norm{n}_2^2 \norm{\na(n^2)}_2^{3/2}\\
\le&~ \ve \norm{\na(n^2)}_2^2 + C \norm{n}_2^8.
}
Choosing $\ve>0$ sufficiently small, we have $\frac{d}{dt} \norm{n(t)}_4^4 \le C\norm{n}_2^8$, and thus for $t\in(t_0,T)$, where $t_0\in(0,T)$ is fixed, that
\[
\norm{n(t)}_4^4 \le \norm{n(t_0)}_4^4 + C\bke{\underset{s\in(t_0,T)}{\sup}\norm{n(s)}_2^8} t.
\]
Therefore, by \eqref{prop5.1-step1}, we get for any $t_0\in(0,T)$
\EQ{\label{prop5.1-step2}
\sup_{t\in(t_0,T)}\norm{n(t)}_4<\infty
}
if $T<\infty$.

\medskip

\noindent {\bf Step 3.} 
Since $\na\cdot u=0$, $\norm{\na u}_2 = \norm{\na\times u}_2 = \norm{\omega}_2$, where $\omega = \na \times u$. The vorticity equation of $\eqref{PKS-NS}_3$ reads 
\[
\pd_t \omega + (u \cdot \na) \omega = \De \omega + \na \times(n \na c),
\]
which has the following energy equality:
\EQN{
\frac12\, \frac{d}{dt}\, \norm{\omega(t)}_2^2 + \norm{\na \omega (t)}_2^2 
=& \int_{\R^2} \na \times(n \na c)\, \omega\, dx\\
=& - \int_{\R^2} (n \na c)\cdot (\na^\perp \omega)\, dx,
}
where we integrate by parts in the last equality. Thus, for $t\in(t_0, T)$, where $t_0\in(0,T)$ is fixed,
\EQN{
\frac12\, \norm{\omega(t)}_2^2 + \int_{t_0}^t \norm{\na \omega (\tau)}_2^2\, d\tau
=&~\frac12\, \norm{\omega(t_0)}_2^2 -  \int_{t_0}^t\int_{\R^2} (n \na c)\cdot (\na^\perp \omega)(\tau)\, dx d\tau\\
\le&~\frac12\, \norm{\omega(t_0)}_2^2 +  \int_{t_0}^t \norm{(n \na c)(\tau)}_2\norm{\na^\perp\omega(\tau)}_2\, d\tau\\
\le&~\frac12\, \norm{\omega(t_0)}_2^2 +  \int_{t_0}^t \bke{\frac{\norm{(n \na c)(\tau)}_2^2}2 + \frac{\norm{\na^\perp\omega(\tau)}_2^2}2} d\tau.
}
By \eqref{prop5.1-step1} and \eqref{prop5.1-step2},
\EQN{
\sup_{t\in(t_0,T)} \norm{n\na c}_2
\le \sup_{t\in(t_0,T)} \norm{n}_4 \norm{\na c}_4
&\lec \sup_{t\in(t_0,T)} \norm{n}_4 \norm{n}_{4/3}\\
&\le \sup_{t\in(t_0,T)} \norm{n}_4 \norm{n}_1^{1/2} \norm{n}_2^{1/2}
\le C
}
and $\norm{\na^\perp\omega}_2\le \norm{\na\omega}_2$.
We have
\EQN{
\norm{\omega(t)}_2^2 + \frac12\int_{t_0}^t \norm{\na \omega (\tau)}_2^2\, d\tau
\le&~\norm{\omega(t_0)}_2^2 +  \frac{C^2}2\, T.
}
Therefore, we conclude that for any $t_0\in(0,T)$
\EQ{\label{prop5.1-step3}
\sup_{t\in(t_0,T)} \norm{\na u(t)}_2 <\infty
}
if $T<\infty$.

\medskip

\noindent {\bf Step 4.}
By \eqref{prop5.1-step3} and the boundedness $\norm{u}_2$ in \thref{u-a-priori} where the bound is independent of time since $M=8\pi$, we have for $1\le q<\infty$ and $t_0\in(0,T)$
\EQ{\label{prop5.1-step4-0}
\sup_{t\in(t_0,T)} \norm{u(t)}_q \lec \sup_{t\in(t_0,T)} \norm{u(t)}_2^{8/q} \norm{\na u(t)}_2^{8/q} \lec C
}
by the Sobolev's embedding theorem.

For $\tau\in(0,T-t_0)$, where $t_0\in(0,T)$ is fixed,
\[
u(t_0+\tau) = e^{t_0 \De} u(\tau) - \int_0^{t_0} e^{(t_0 - s)\De} {\bf P} \na \cdot \bkt{\na c(s+\tau) \otimes \na c(s+\tau) + u(s+\tau)\otimes u(s+\tau)} ds.
\]
Then
\EQN{
\norm{u(t_0+\tau)}_\infty 
&\le \norm{e^{t_0 \De}u(\tau)}_\infty + C \int_0^{t_0} (t_0 - s)^{-\frac12 - \frac13} \bke{\norm{\na c(s+\tau)}_{6}^2 + \norm{u(s+\tau)}_{6}^2} ds\\
&\lec t_0^{-\frac12} \norm{u_0}_2 + t_0^{\frac16} \sup_{s\in(0,t_0)} \bke{\norm{n(s+\tau)}_{3/2}^2 + \norm{u(s+\tau)}_6^2}\\
&\lec t_0^{-\frac12} \norm{u_0}_2 + t_0^{\frac16} \sup_{\si\in(\tau,T)} \bke{\norm{n(\si)}_1^{2/3} \norm{n(\si)}_2^{4/3} + \norm{u(\si)}_6^2}\\
&\le C
}
by \eqref{prop5.1-step1} and \eqref{prop5.1-step4-0}. Note that the bound is independent of $\tau$ since $\norm{u(t_0+\tau)}_\infty$ is bounded as $\tau\to0_+$. Thus, we conclude for any $t_0\in(0,T)$
\EQ{\label{prop5.1-step4}
\sup_{t\in(t_0,T)} \norm{u(t)}_\infty < \infty
}
if $T<\infty$.

\medskip

\noindent {\bf Step 5.}
From \thref{Nagai2011-lem2.5}, we have for any $t_0\in(0,T)$
\EQ{\label{prop5.1-step5-1}
\sup_{t\in(t_0,T)} \norm{\na c(t)}_\infty \lec \sup_{t\in(t_0,T)} \norm{n(t)}_1^{\frac13} \norm{n(t)}_4^{\frac23} \le C
}
by \eqref{prop5.1-step2}. For $\tau\in(0,T-t_0)$, where $t_0\in(0,T)$ is fixed,
\[
n(t_0+\tau) = e^{t_0 \De} n(\tau) - \int_0^{t_0} e^{(t_0 - s)\De} \na \cdot \bkt{n(s+\tau)\na c(s+\tau) + n(s+\tau) u(s+\tau)} ds.
\]
Then
\EQN{
&\norm{n(t_0+\tau)}_\infty \\
&\le \norm{e^{t_0 \De}n(\tau)}_\infty + C \int_0^{t_0} (t_0 - s)^{-\frac12 - \frac14} \bke{\norm{n(s+\tau)\na c(s+\tau)}_4 + \norm{n(s+\tau)u(s+\tau)}_4} ds\\
&\lec t_0^{-1} \norm{n(\tau)}_1 + t_0^{\frac14} \bkt{\sup_{\si\in(\tau,T)} \bke{\norm{\na c(\si)}_\infty + \norm{u(\si)}_\infty}} \sup_{\si\in(\tau,T)} \norm{n(\si)}_4\\
&\le C
}
by \eqref{prop5.1-step5-1}, \eqref{prop5.1-step4} and \eqref{prop5.1-step2}. Note that the bound is independent of $\tau$ since $\norm{n(t_0+\tau)}_\infty$ is bounded as $\tau\to0_+$. Thus, we conclude for any $t_0\in(0,T)$
\EQ{\label{prop5.1-step5}
\sup_{t\in(t_0,T)} \norm{n(t)}_\infty < \infty
}
if $T<\infty$.
The proposition follows from \eqref{prop5.1-step4} and \eqref{prop5.1-step5}.
\end{proof}

\medskip

\begin{proof}[Proof of \thref{NO-thm1.2}]

Let $T_m$ be the maximal existence time of $(n,u)$, and suppose $T_m<\infty$ by contradiction. 

\medskip

For the subcritical case $M<8\pi$. Let $T<T_m$ be fixed. From (iv) of \thref{local-exist}, $n(T)$ and $u(T)$ belong to $H^2(\R^2)$. According to \cite[Theorem 1]{GH-arxiv2020}, solutions exist globally in time, which contradicts the definition of $T_m$. Moreover, \eqref{eq-Linf-bound} is a consequence of \cite[(1.3)]{GH-arxiv2020} and a use of Sobolev's embedding.

\medskip

For the critical case $M=8\pi$, we first claim that $n(t)\to\tilde n_0$ in $L^{4/3}$ and $u(t)\to\tilde u_0$ in $L^2$ as $t\to T_m-$ for some $(\tilde n_0, \tilde u_0)\in L^{4/3}\times L^2_\si$. Indeed, for fixed $t_0\in(0,T_m)$, since
\EQN{
\na^k n(t) &= \na^k e^{(t-t_0/2)\De}\, n(t_0/2) - \int_{t_0}^t \na^k \na e^{(t-s)\De} \bke{n(s)\na c(s) + n(s) u(s)} ds\\
&=: \na^k e^{(t-t_0/2)\De}\, n(t_0/2) + I,
}
where, by \cite[(3.1)]{Nagai-2011},
\[
\norm{\na^k e^{(t-t_0/2)\De}\, n(t_0/2)}_{4/3}\le Ct_0^{-k/2}\norm{n(t_0/2)}_{4/3}
\]
and
\EQN{
\norm{I}_{4/3} 
&\lec C \int_{t_0}^t (t-s)^{-\bke{1-\frac34}-\frac{k+1}2} \norm{n(s)\bke{\na c(s) + u(s)}}_1\, ds\\
&\lec C \int_{t_0}^t (t-s)^{-\frac34-\frac{k}2} \norm{n(s)}_{4/3}\bke{\norm{n(s)}_{4/3}+\norm{u(s)}_4}\, ds
}
in which $\norm{n}_{4/3}$ and $\norm{u}_4$ are uniformly bounded in time by \thref{prop-Linfty-bound}, we obtain for $0<k<1/2$ that
\[
\sup_{t_0\le t\le T_m}\norm{\na^k n(t)}_{4/3}<\infty,
\]
and thus, using \cite[(3.2)]{Nagai-2011},
\[
\norm{e^{(t-\tau)\De}n(\tau) - n(\tau)}_{4/3} \le C(t-\tau)^{k/2} \norm{\na^k n(\tau)}_{4/3} \to0
\]
as $\tau,t\to T_m-$. Therefore, $n_0(t)\to\tilde n_0$ in $L^{4/3}$ as $t\to T_m-$. Similarly, since 
\[
\norm{e^{(t-\tau)\De} u(\tau) - u(\tau)}_4 \le C(t-\tau)^{3/4} \norm{\na u(\tau)}_2
\]
in which $\sup_{t\in(t_0,T)} \norm{\na u(t)}_2 <\infty$ by \eqref{prop5.1-step3}, we also have $u_0\to\tilde u_0$ in $L^4$ as $t\to T_m-$. It is obvious that $\tilde n_0\in L^1\cap L^{4/3}$ and $\tilde u_0\in L^2\cap L^4$. By \thref{local-exist} and \thref{Nagai-2011-rmk2.1}, there exists a unique mild solution $(\tilde n,\tilde u)$ on $[0,T_0)$, $T_0>0$, with initial data $(\tilde n_0, \tilde u_0)$, satisfying $\tilde n\in BC([0,T_0);L^1\cap L^{4/3})$ and $\tilde u\in BC([0,T_0);L^2\cap L^4)$. Define $(\hat n,\hat u)$ by
\[
\hat n(t) = 
\begin{cases}
n(t),&\qquad 0\le t<T_m,\\
\tilde n(t-T_m),&\qquad T_m\le t<T_m+T_0,
\end{cases}
\]
\[
\hat u(t) = 
\begin{cases}
u(t),&\qquad 0\le t<T_m,\\
\tilde u(t-T_m),&\qquad T_m\le t<T_m+T_0.
\end{cases}
\]
Then $(\hat n,\hat u)$ is a mild solution on $[0,T_m+T_0)$, contradicting the choice of $T_m$. This proves the global-in-time existence. Finally, the $L^\infty$-boundedness \eqref{eq-Linf-bound} follows from \thref{prop-Linfty-bound}.
\end{proof}

\bigskip

\section*{Acknowledgements}
The research of C. Lai is partially supported by FYF (\#6456) of Graduate and Postdoctoral Studies, UBC. The research of J. Wei is partially supported by the NSERC grant of Canada. C. Lai thanks Prof. Tai-Peng Tsai for useful discussions and finding a typo.

\medskip







\def\cprime{$'$}

\end{document}